\documentclass[11pt,a4paper,reqno]{amsart}

\usepackage[a4paper,lmargin=2cm,rmargin=2cm,tmargin=4cm,bmargin=4cm]{geometry}

\usepackage[english]{babel}

\usepackage[centertags]{amsmath}
\usepackage{amsfonts}
\usepackage{amssymb}
\usepackage{amsthm}
\usepackage{dsfont}

\usepackage[dvipsnames]{xcolor}

\usepackage[normalem]{ulem}
\usepackage{bbm}

\usepackage{hyperref}
	\hypersetup{breaklinks=true,colorlinks=true,
linkcolor=MidnightBlue,citecolor=MidnightBlue,
urlcolor=MidnightBlue}
\usepackage{enumerate}

\theoremstyle{plain}
\newtheorem{theorem}{Theorem}[section]
\newtheorem{corollary}[theorem]{Corollary}
\newtheorem{proposition}[theorem]{Proposition}
\newtheorem{lemma}[theorem]{Lemma}

\theoremstyle{definition}
\newtheorem{definition}[theorem]{Definition}

\newtheorem{example}[theorem]{Example}
\newtheorem{examples}[theorem]{Examples}
\newtheorem{remark}[theorem]{Remark}
\newtheorem{question}[theorem]{Question}

\newcommand{\N}{\ensuremath{\mathbb{N}}}
\newcommand{\C}{\ensuremath{\mathbb{C}}}
\newcommand{\K}{\ensuremath{\mathbb{K}}}
\newcommand{\R}{\ensuremath{\mathbb{R}}}

\newcommand{\T}{\ensuremath{\mathbb{T}}}
\newcommand{\eps}{\ensuremath{\varepsilon}}

\newcommand{\ext}{\operatorname{ext}}
\newcommand{\extm}{\operatorname{extm}}
\newcommand{\conv}{\operatorname{conv}}
\newcommand{\aconv}{\operatorname{aconv}}
\newcommand{\att}{\operatorname{att}}
\newcommand{\dent}{\operatorname{dent}}
\newcommand{\Slice}{\operatorname{Slice}}

\newcommand{\Face}{\operatorname{Face}}
\newcommand{\re}{\operatorname{Re}}
\newcommand{\Id}{\operatorname{Id}}
\newcommand{\Gen}{\operatorname{Gen}}
\newcommand{\Spear}{\operatorname{Spear}}
\newcommand{\diam}{\operatorname{diam}}
\newcommand{\dens}{\operatorname{dens}}
\newcommand{\dist}{\operatorname{dist}}

\renewcommand{\leq}{\leqslant}
\renewcommand{\geq}{\geqslant}

\begin{document}
\title[Generating operators between Banach space]{Generating operators between Banach spaces}
\author[Kadets]{Vladimir Kadets}
\address[Kadets]{Holon Institute of Technology\\ School of Mathematical Sciences\\  52 Golomb Street\\ POB 305 Holon 5810201\\ Israel \newline
\href{http://orcid.org/0000-0002-5606-2679}{ORCID: \texttt{0000-0002-5606-2679}}}
\email{vova1kadets@yahoo.com}

\author[Mart\'{\i}n]{Miguel Mart\'{\i}n}
	\address[Mart\'{\i}n]{Universidad de Granada \\ Facultad de Ciencias \\
		Departamento de An\'{a}lisis Matem\'{a}tico \\ E-18071 Granada \\
		Spain \newline
		\href{http://orcid.org/0000-0003-4502-798X}{ORCID: \texttt{0000-0003-4502-798X} }}
	\email{mmartins@ugr.es}
	\urladdr{\url{https://www.ugr.es/local/mmartins}}

\author[Mer\'{\i}]{Javier Mer\'{\i}}
	\address[Mer\'{\i}]{Universidad de Granada \\ Facultad de Ciencias \\
		Departamento de An\'{a}lisis Matem\'{a}tico \\ E-18071 Granada \\
		Spain \newline
		\href{http://orcid.org/0000-0002-0625-5552}{ORCID: \texttt{0000-0002-0625-5552} }}
	\email{jmeri@ugr.es}

\author[Quero]{Alicia Quero}
	\address[Quero]{Universidad de Granada \\ Facultad de Ciencias \\
		Departamento de An\'{a}lisis Matem\'{a}tico \\ E-18071 Granada \\
		Spain \newline
		\href{http://orcid.org/0000-0003-4534-8097}{ORCID: \texttt{0000-0003-4534-8097} }}
	\email{aliciaquero@ugr.es}

\date{June 5th, 2023}

\begin{abstract}
	We introduce and study the notion of generating operators as those norm-one operators $G\colon X\longrightarrow Y$ such that for every $0<\delta<1$, the set $\{x\in X\colon \|x\|\leq 1,\ \|Gx\|>1-\delta\}$ generates the unit ball of $X$ by closed convex hull. This class of operators includes isometric embeddings, spear operators (actually, operators with the alternative Daugavet property), and other examples like the natural inclusions of $\ell_1$ into $c_0$ and of $L_\infty[0,1]$ into $L_1[0,1]$. We first present a characterization in terms of the adjoint operator, make a discussion on the behaviour of diagonal generating operators on $c_0$-, $\ell_1$-, and $\ell_\infty$-sums, and present examples in some classical Banach spaces. Even though rank-one generating operators always attain their norm, there are generating operators, even of rank-two, which do not attain their norm. We discuss when a Banach space can be the domain of a generating operator which does not attain its norm in terms of the behaviour of some spear sets of the dual space. Finally, we study when the set of all generating operators between two Banach spaces $X$ and $Y$ generates all non-expansive operators by closed convex hull. We show that this is the case when $X=L_1(\mu)$ and $Y$ has the Radon-Nikod\'{y}m property with respect to $\mu$. Therefore, when $X=\ell_1(\Gamma)$, this is the case for every target space $Y$. Conversely, we also show that a real finite-dimensional space $X$ satisfies that generating operators from $X$ to $Y$ generate all non-expansive operators by closed convex hull only in the case that $X$ is an $\ell_1$-space.
\end{abstract}


\maketitle

\section{Introduction}
Let $X$ and $Y$ be Banach spaces over the field $\K$ ($\K=\R$ or $\K=\C$). We denote by $\mathcal{L}(X,Y)$ the space of all bounded linear operators from $X$ to $Y$ and write $X^*=\mathcal{L}(X,\K)$ to denote the dual space. By $B_X$ and $S_X$ we denote the closed unit ball and the unit sphere of $X$, respectively, and we write $\T$ for the set of modulus one scalars. Some more notation and definitions (which are standard) are included in Subsection~\ref{Subsect:notation} at the end of this introduction.

The concept of spear operator was introduced in \cite{Ardalani} and deeply studied in the book \cite{KMMP-Spear}. A norm-one operator $G\in \mathcal{L}(X,Y)$ is said to be an \emph{spear operator} if the norm equality
$$
\max_{\theta\in \T}\|G+\theta T\|=1+\|T\|
$$
holds for all $T\in \mathcal{L}(X,Y)$.
This concept extends the properties of the identity operator in those Banach spaces having numerical index one and it is satisfied, for instance, by the Fourier transform on $L_1$. There are isometric and isomorphic consequences on the domain and range spaces of a spear operator as, for instance, in the real case, the dual of the domain of a spear operator with infinite rank has to contain a copy of $\ell_1$. For more information and background, we refer the interested reader to the already cited book \cite{KMMP-Spear}. Even though the definition of spear operator given above does not need numerical ranges, it is well known that spear operators are exactly those operators such that the numerical radius with respect to them coincides with the operator norm. Let us introduce the relevant definitions. Fixed a norm-one operator $G\in \mathcal{L}(X,Y)$, the \emph{numerical radius with respect to $G$} is the seminorm defined as
\begin{align*}
v_G(T)&:=\sup\{|\phi(T)|\colon \phi\in \mathcal{L}(X,Y)^*,\,\phi(G)=1\}\\
&=\inf_{\delta>0}\sup\{|y^*(Tx)|\colon y^*\in S_{Y^*},\,x\in S_X,\, \re y^*(Gx)>1-\delta\}
\end{align*}
for every $T\in \mathcal{L}(X,Y)$ (the equality above was proved in \cite[Theorem~2.1]{Mar-numrange-JMAA2016}). Observe that $v_G(\cdot)$ is a seminorm in $\mathcal{L}(X,Y)$ which clearly satisfies
\begin{equation}\label{eq:vGleqnorm}
v_G(T)\leq \|T\| \qquad (T\in \mathcal{L}(X,Y)).
\end{equation}
Then, $G$ is a spear operator if and only if $v_G(T)=\|T\|$ for every $T\in \mathcal{L}(X,Y)$ (see \cite[Proposition~3.2]{KMMP-Spear}).

Our discussion here starts with the observation that it is possible to introduce a natural seminorm between $v_G(T)$ and $\|T\|$ in Eq.~\eqref{eq:vGleqnorm}: the (semi-)norm relative to $G$. Let us introduce the needed notation and definitions. Let $X$, $Y$, $Z$ be Banach spaces and let $G\in \mathcal{L}(X,Y)$ be a norm-one operator. For $\delta>0$, we write $\att(G,\delta)$ to denote the \emph{$\delta$-attainment set} of $G$, that is,
$$
\att(G,\delta):=\{x\in S_X \colon \|Gx\|>1-\delta\}.
$$
If there exists $x\in S_X$ such that $\|Gx\|=1$, we say that $G$ \emph{attains its norm} and we denote by $\att(G)$ the attainment set of $G$:
$$
\att(G):=\{x\in S_X\colon \|Gx\|=1\}.
$$
We consider the parametric family of norms on $\mathcal{L}(X,Z)$ defined by
$$
\|T\|_{G,\delta}:=\sup \left\{ \|Tx\|\colon x\in \att(G,\delta) \right\} \qquad (T\in\mathcal{L}(X,Z))
$$
which are equivalent to the usual norm on $\mathcal{L}(X,Z)$ (this is so since $\att(G,\delta)$ has nonempty interior). We are interested in the (semi-)norm obtained taking infimum on this parametric family.

\begin{definition}
Let $X$, $Y$ and $Z$ be Banach spaces and let $G\in\mathcal{L}(X,Y)$ be a norm-one operator. For $T\in \mathcal{L}(X,Z)$, we define the \emph{(semi-)norm of $T$ relative to $G$} by
$$
\|T\|_G:=\inf_{\delta>0} \|T\|_{G,\delta}.
$$
\end{definition}

When $Z=Y$, we clearly have that
$$
v_G(T)\leq \|T\|_G\leq \|T\| \qquad (T\in \mathcal{L}(X,Y))
$$
and so this $\|\cdot\|_G$ is the promised seminorm to extend Eq.~\eqref{eq:vGleqnorm}. We may study the possible equality between $v_G(\cdot)$ and $\|\cdot\|_G$ and between $\|\cdot\|_G$ and the usual operator norm. We left the first relation for a subsequent paper which is still in process \cite{RelativeSpears}. The main aim in this manuscript is to study when the norm equality
\begin{equation}\label{eq:generating-def-eq-n_G-n}
\|T\|_G=\|T\|
\end{equation}
holds true.

\begin{definition}
Let $X$, $Y$ be Banach spaces. We say that $G\in \mathcal{L}(X,Y)$ with norm-one is \emph{generating} (or a \emph{generating operator}) if equality \eqref{eq:generating-def-eq-n_G-n} holds true for all $T\in \mathcal{L}(X,Y)$. We denote by $\Gen(X,Y)$ the set of all generating operators from $X$ to $Y$.
\end{definition}

Observe that both $\|\cdot\|_G$ and the operator norm can be defined for operators with domain $X$ and arbitrary range, so one may wonder if there are different definitions of generating requiring that Eq.~\eqref{eq:generating-def-eq-n_G-n} holds replacing $Y$ for other range spaces. This is not the case, as we will show in Section~\ref{Sect-characterization-first-results} that a generating operator $G$ satisfies that $\|T\|_G=\|T\|$ for every $T\in \mathcal{L}(X,Z)$ and every Banach space $Z$ (see Corollary~\ref{corollary:characterization-generating}). This is so thanks to a characterization of generating operators in terms of the sets $\att(G,\delta)$: $G$ is generating (if and) only if $\overline{\conv}(\att(G,\delta)) = B_X$ for every $\delta>0$, see  Corollary~\ref{corollary:characterization-generating} again. When the dimension of $X$ is finite, this is clearly equivalent to the fact that $\conv(\att(G))=B_X$ (actually, the same happens for compact operators defined on reflexive spaces, see Proposition~\ref{prop:X-reflexive-G-compact}). For some infinite-dimensional $X$, there are generating operators from $X$ which do not attain their norm, even of rank-two (see Example~\ref{example:generating-rank-two-NONA}); but there are even generating operators attaining the norm such that $\overline{\conv}(\att(G))$ has empty interior (see Example~\ref{example:generatingNonNA-convattemptyinterior}).

There is another characterization which involves the geometry of the dual space. We need some definitions. A subset $F$ of the unit ball of a Banach space $Z$ is said to be a \emph{spear set} of $Z$ \cite[Definition~2.3]{KMMP-Spear} if
$$
\max_{\theta \in \T}\sup_{z\in F}\|z+\theta x\|=1 + \|x\| \qquad (x\in Z).
$$
If $z\in S_Z$ satisfies that $F=\{z\}$ is a spear set, we just say that $z$ is a \emph{spear vector} and we write $\Spear(Z)$ for the set of spear vectors of $Z$. We refer the reader to \cite[Chapter 2]{KMMP-Spear} for more information and background. We will show that a norm-one operator $G\in \mathcal{L}(X,Y)$ is generating if and only if $G^*(B_{Y^*})$ is a spear set of $X^*$, see Corollary~\ref{corollary:charact-generating-dual-spear}. These characterizations appear in Section~\ref{Sect-characterization-first-results}, together with a discussion on the behaviour of diagonal generating operators on $c_0$-, $\ell_1$-,
and $\ell_\infty$-sums, and examples in some classical Banach spaces.

We next discuss in Section~\ref{Sect-normattaiment} the relationship between generating operators and norm attainment. On the one hand, we show that rank-one generating operators attain their norm (see Corollary~\ref{cor:x*-generating-attains-norm}) and, clearly, the same happens with isometric embeddings (which are generating), or with generating operators whose domain has the RNP (see Corollary~\ref{cor:domain-RNP}), as every generating operator attains its norm on denting points (see Lemma~\ref{lemma:denting-points}). But, on the other hand, there are generating operators, even of rank two, which do not attain their norm  (see Example~\ref{example:generating-rank-two-NONA}). We further discuss the possibility for a Banach space $X$ to be the domain of a generating operator which does not attain its norm in terms of the behaviour of some spear sets of $X^*$ (see Theorem~\ref{theorem:char-when-X-domain-generating-NONA}).

Finally, Section~\ref{section:setGen(X,Y)} is devoted to the study of the set $\Gen(X,Y)$. We show that it is closed (see Proposition~\ref{prop:GenXY-norm-closed}), and show that for every Banach space $Y$, there is a Banach space $X$ such that $\Gen(X,Y)=\emptyset$ (see Proposition~\ref{proposition:foreveryYthereisXgenempty}), but this result is not true for $Y=C[0,1]$ if we restrict the space $X$ to be separable (Example~\ref{example:Y=C[01]everyXseparableGennotempty}). We next study properties of $\Gen(X,Y)$ when $X$ is fixed. We first show that $\Gen(X,Y)\neq \emptyset$ for every $Y$ if and only if $\Spear(X^*)\neq \emptyset$ (see Corollary~\ref{corollary:GenXYnotemptyforallYiffspearXstarnotempty}) and that the only case in which there is $Y$ such that $\Gen(X,Y)=S_{\mathcal{L}(X,Y)}$ is when $X$ is one-dimensional (see Corollary~\ref{corollary:XonedimensionaliffGenXYequalall}). We then study the possibility that the set $\Gen(X,Y)$ generates the unit ball of $\mathcal{L}(X,Y)$ by closed convex hull, showing first that this is the case when
$X=L_1(\mu)$ and $Y$ has the RNP (Theorem~\ref{theorem:representable-in-Gen}) and when
$X=\ell_1(\Gamma)$ and $Y$ is arbitrary (see Proposition~\ref{cor:conv-Gen-l1}) and that this is the only possibility for \emph{real} finite-dimensional spaces (see Proposition~\ref{thm:conv-Gen-X-finite-ext}).

\subsection{A bit of notation}\label{Subsect:notation}
Let $X$, $Y$ be Banach spaces. We write $J_X\colon X\longrightarrow X^{**}$ to denote the natural inclusion of $X$ into its bidual space. A subset $\mathcal{A}\subseteq B_{X^*}$ is \emph{$r$-norming for $X$} ($0<r\leq 1$) if $rB_{X^*}\subseteq \overline{\aconv}^{w^*}(\mathcal{A})$ or, equivalently, if $r\|x\|\leq \sup_{x^*\in \mathcal{A}}|x^*(x)|$ for every $x\in Z$. The most interesting case is $r=1$: $\mathcal{A}$ is \emph{one-norming for $X$} if $B_{X^*}= \overline{\aconv}^{w^*}(\mathcal{A})$ or, equivalently, if $\|x\|= \sup_{x^*\in \mathcal{A}}|x^*(x)|$ for every $x\in Z$. A \emph{slice} of a closed convex bounded set $C\subset X$ is a nonempty intersection of $C$ with an open half-space. We write
$$
\Slice(C,f,\alpha):=\left\{x\in C\colon \re f(x)>\sup_{C}\re f - \alpha\right\}
$$
where $f\in X^*$ and $\alpha>0$, and observe that every slice of $C$ is of the above form.

For $A\subset X$, $\conv(A)$ and $\aconv(A)$ are, respectively, the convex hull and the absolutely convex hull of $A$; $\overline{\conv}(A)$ and $\overline{\aconv}(A)$ are, respectively, the closures of these sets. For $B\subset X$ convex, $\ext(B)$ denotes the set of extreme points of $B$.

\section{Characterizations, first results, and some examples} \label{Sect-characterization-first-results}
Our first result gives different characterizations for the equivalence of $\|\cdot\|$ and $\|\cdot\|_{G}$ on $\mathcal{L}(X,Z)$. As one may have expected, this does not depend on the range space $Z$.

\begin{proposition}\label{prop:characterization-rgenerating}
Let $X$, $Y$ be Banach spaces, let $G\in\mathcal{L}(X,Y)$ be a norm-one operator, and let $r\in (0,1]$.
Then, the following are equivalent:
\begin{enumerate}[$(i)$]
\item $\|T\|_G\geq r \|T\|$ for every Banach space $Z$ and every $T\in\mathcal{L}(X,Z)$.
\item There is a (non null) Banach space $Z$ such that $\|T\|_G\geq r \|T\|$ for every $T\in\mathcal{L}(X,Z)$.
\item There is a (non null) Banach space $Z$ such that $\|T\|_G\geq r \|T\|$ for every rank-one operator $T\in\mathcal{L}(X,Z)$.
\item $\|x^*\|_G\geq r\|x^*\|$ for every $x^*\in X^*$.
\item $\|x^*\|_{G,\delta}\geq r\|x^*\|$ for every $x^*\in X^*$ and every $\delta>0$.
\item $\overline{\conv}(\att(G,\delta))\supseteq rB_X$ for every $\delta>0$.
\end{enumerate}
\end{proposition}

\begin{proof}\parskip=0ex
	The implications $(i)\Rightarrow(ii)\Rightarrow(iii)$, $(iv)\Leftrightarrow(v)$, and $(vi)\Rightarrow(i)$ are evident.
	
	$(iii)\Rightarrow(iv)$. Fix $z\in S_Z$ and, given $x^*\in X^*$, consider $T=x^*\otimes z \in \mathcal{L}(X,Z)$ which obviously satisfies $\|T\|=\|x^*\|$ and $\|T\|_G=\|x^*\|_G$.
	
	The remaining implication $(v)\Rightarrow(vi)$ follows from the Bipolar theorem. Indeed, for $\delta>0$, take $x\in rB_X$, we have to prove that $J_X(x)$ belongs to $\att(G,\delta)^{\circ \circ}$. For $x^*\in \att(G,\delta)^{\circ}$,
	$$
	|J_X(x)(x^*)|=|x^*(x)|\leq r\|x^*\|\leq \|x^*\|_G\leq \sup\left\{|x^*(x)|\colon x\in \att(G,\delta)\right\}\leq 1,
	$$
	where the second inequality follows from $(iv)$ and the last one from the fact that $x^*\in \att(G,\delta)^{\circ}$. Therefore $J_X(x)\in \att(G,\delta)^{\circ\circ}=\overline{\conv}^{w^*}(\att(G,\delta))$.
\end{proof}

Observe that item $(vi)$ in the previous result just means that, for every $\delta\in (0,1)$, the set $\att(G,\delta)$ is $r$-norming for $X^*$. This leads to the following concept which extends the one of generating operator.

\begin{definition}
Let $X$, $Y$ be Banach spaces, let $G\in\mathcal{L}(X,Y)$ be a norm-one operator and let $r\in (0,1]$. We say that $G$ is \emph{$r$-generating} if $\overline{\conv}(\att(G,\delta))\supseteq rB_X$ for every $\delta>0$.
\end{definition}

Of course, the case $r=1$ coincides with the generating operators introduced in the introduction. For them, the following characterization deserves to be emphasized.

\begin{corollary}\label{corollary:characterization-generating}
Let $X$, $Y$ be Banach spaces, let $G\in\mathcal{L}(X,Y)$ be a norm-one operator.
Then, the following are equivalent:
\begin{enumerate}[$(i)$]
\item $G$ is generating.
\item $\|T\|_G=\|T\|$ for every $T\in\mathcal{L}(X,Z)$ and every Banach space $Z$.
\item There is a (non null) Banach space $Z$ such that $\|T\|_G =\|T\|$ for every rank-one operator $T\in\mathcal{L}(X,Z)$.
\item $B_X=\overline{\conv}(\att(G,\delta))$ for every $\delta>0$.
\end{enumerate}
In particular, if there exists  $A\subseteq B_X$ which satisfies $\overline{\aconv}(A)=B_X$ and $A\subseteq \att(G,\delta)$ for every $\delta>0$, then $G$ is generating.
\end{corollary}

In the next list we give the first easy examples of generating operators.

\begin{examples}\label{examples:preliminaryones}
$ $
\begin{enumerate}[(1)]
\item {\slshape The identity operator on every Banach space is generating.}
\item {\slshape Actually, all isometric embeddings are generating.}
\item {\slshape Spear operators are generating since, in this case, $v_G(T)=\|T\|$ for every $T\in\mathcal{L}(X,Y)$.}
\item {\slshape Actually, operators with the \emph{alternative Daugavet property} (i.e.\ those $G\in \mathcal{L}(X,Y)$ such that $v_G(T)=\|T\|$ for every $T\in \mathcal{L}(X,Y)$ with rank-one, cf.\ \cite[Section 3.2]{KMMP-Spear}) are also generating by using Corollary~\ref{corollary:characterization-generating} with $Z=Y$} in item (iii).
\item {\slshape The natural embedding $G$ of $\ell_1$ into $c_0$ is a generating operator.}\newline \indent
Indeed, for every $\delta>0$, we have that
$$
\att(G,\delta)=\left\{x\in S_{\ell_1}\colon \|Gx\|_\infty>1-\delta\right\}\supset \T\{ e_n\colon n\in\N\},
$$
so $\overline{\conv}(\att(G,\delta))=B_{\ell_1}$.
\item {\slshape The natural embedding $G$ of $L_\infty[0,1]$ into $L_1[0,1]$ is a generating operator.}\newline\indent
Indeed, for every $\delta>0$, notice that $B_{L_\infty[0,1]}=\overline{\conv}\left(\{f\in L_\infty[0,1] \colon \left|f(t)\right|=1 \ \textnormal{a.e.} \}\right)$ (this should be well known, but in any case it follows from  Lemma~\ref{lemma:L_infty(mu,Y)} which includes the vector-valued case). Observe then that, for every $f\in L_\infty[0,1]$ satisfying $\left|f(t)\right|=1$ \textnormal{a.e.}, it follows $\|f\|_\infty=\|G(f)\|_1=1$. So $\|G(f)\|_1=1$ and $f\in\att(G,\delta)$.
\end{enumerate}
\end{examples}

We will provide some more examples in classical Banach spaces in Subsection~\ref{subsect:classical}.

The next result deals with compact operators defined on a reflexive Banach space.

\begin{proposition}\label{prop:X-reflexive-G-compact}
	Let $X$ be a reflexive Banach space, let $Y$ be a Banach space, and let $G\in\mathcal{L}(X,Y)$ be a compact operator with $\|G\|=1$. Then,
	$$\bigcap_{\delta>0}\overline{\conv}(\att(G,\delta))=\overline{\conv}(\att(G)).$$
	Consequently, $G$ is $r$-generating if and only if $r B_X\subseteq \overline{\conv(\att(G))}$.
\end{proposition}
\begin{proof}
	Let $x_0\in \bigcap_{\delta>0}\overline{\conv}(\att(G,\delta))$ and suppose that $x_0\notin \overline{\conv}(\att(G))$. Then there exist $x_0^*\in X^*$ and $\alpha> 0$ such that
	\begin{equation}\label{eq:B_G-X-reflexive}
		\sup_{x\in\overline{\conv}(\att(G))} \re x_0^*(x)< \alpha\leq\re x_0^*(x_0).
	\end{equation}
	Fix $\eps>0$. Given $n\in\N$, since $x_0\in \overline{\conv}\left(\att\left(G,\frac{1}{n}\right)\right)$, we may find $m\in\N$, $y_1,\ldots,y_m\in \att\left(G,\frac{1}{n}\right)$, and $\lambda_1,\ldots,\lambda_m\in[0,1]$ with $\sum_{k=1}^m\lambda_k=1$ such that
	$$ \left\|x_0-\sum_{k=1}^m\lambda_k y_k\right\|<\eps,$$
	hence
	$$\alpha-\eps\leq\re x_0^*(x_0)-\eps<\sum_{k=1}^m \lambda_k \re x_0^*(y_k).$$
	By convexity, there is $k_0\in\{1,\ldots,m\} $ such that $\re x_0^*(y_{k_0})\geq\alpha-\eps$. Repeating this argument for each $n\in\N$, we obtain a sequence $\{y_n\}$ in $B_X$ such that $\re x_0^*(y_n)>\alpha-\eps$ and $\|Gy_n\|> 1-\frac{1}{n}$ for every $n\in\N$. Now, using that $B_X$ is weakly compact by Dieudonn\'{e}'s theorem, we obtain a subsequence $\{y_{\sigma(n)}\}$ of $\{y_n\}$ which is weakly convergent to some $y_0\in B_X$. Then, by the arbitrariness of $\eps$ and the compactness of $G$ we have that $$\re x_0^*(y_0)\geq\alpha \qquad \textnormal{and} \qquad \|Gy_0\|=1,$$
	which contradicts \eqref{eq:B_G-X-reflexive}.
\end{proof}

Clearly, the previous result applies when $X$ is finite-dimensional.

\begin{corollary}\label{cor:char-finite-dimensional}
	Let $X$ be a finite-dimensional space, let $Y$ be a Banach space, and let $G\in\mathcal{L}(X,Y)$ with $\|G\|=1$. Then,
	$$\bigcap_{\delta>0}\overline{\conv}(\att(G,\delta))=\conv(\att(G)).$$
	Consequently, $G$ is $r$-generating if and only if $r B_X\subseteq \conv(\att(G))$.
\end{corollary}

The next result characterizes those operators acting from a finite-dimensional space which are $r$-generating for some $0<r\leq 1$.

\begin{proposition}\label{prop:characterization-rGenerating-Xfindim}
Let $X$ be a Banach space with $\dim(X)=n$, let $Y$ be a Banach space, and let $G\in \mathcal{L}(X,Y)$ with $\|G\|=1$. The following are equivalent:
\begin{enumerate}[$(i)$]
\item $G$ is $r$-generating for some $r\in (0,1]$.
\item The set $\att(G)$ contains $n$ linearly independent elements.
\end{enumerate}
\end{proposition}

\begin{proof}
$(i)\Rightarrow (ii)$. By Corollary~\ref{cor:char-finite-dimensional}, we have that
$rB_X\subseteq \conv(\att(G)).$ Therefore, $\att(G)$ contains $n$ linearly independent elements.

$(ii)\Rightarrow (i)$. We start proving that the set $\conv(\att(G))$ is absorbing. Indeed, let $\{x_1,\ldots, x_n\}$ be a linearly independent subset of $\att(G)$. Then, fixed $0\neq x\in X$, there are $\lambda_1,\ldots,\lambda_n \in\K$ such that $x=\sum_{j=1}^n\lambda_j x_j$. Calling $0<\rho=\sum_{j=1}^n|\lambda_j|$ we can write
\begin{align*}
x=\sum_{k=1}^n\lambda_j x_j &= \sum_{\lambda_j\neq 0} |\lambda_j|\frac{\lambda_j}{|\lambda_j|} x_j= \rho \sum_{\lambda_j\neq 0} \frac{|\lambda_j|}{\rho}\frac{\lambda_j}{|\lambda_j|} x_j \in \rho \conv(\att(G))
\end{align*}
where we used that $\frac{\lambda_j}{|\lambda_j|} x_j \in\att(G)$ as this set is balanced.
Hence, the set $\conv(\att(G))$ is absorbing. Besides, $\conv(\att(G))$ is clearly balanced, convex, and compact. So its Minkowski functional defines a norm on $X$ which must be equivalent to the original one. Then, there is $r>0$ such that $rB_X\subseteq \conv(\att(G))$ and, therefore, $G$ is $r$-generating by Corollary~\ref{cor:char-finite-dimensional}.
\end{proof}

We next would like to present the relationship of generating operators with denting points (and so with the Radon-Nikod\'{y}m property, RNP in short). We need some notation. Let $A$ be a bounded closed convex set. Recall that $x_0\in A$ is a \emph{denting point} if for every $\delta>0$ $x_0\notin\overline{\conv}(A\setminus B(x_0,\delta))$ or, equivalently, if $x_0$ belongs to slices of $B_X$ of arbitrarily small diameter. We write $\dent(A)$ to denote the set of denting points of $A$. A closed convex subset $C$ of $X$ has the \emph{Radon-Nikod\'{y}m property} (\emph{RNP} in short), if all of its closed convex bounded subsets contain denting points or, equivalently, if all of its closed convex bounded subsets are equal to the closed convex hull of their denting points. In particular, the whole space $X$ may also have this property.

The following result tells us that generating operators must attain their norms on every denting point.

\begin{lemma}\label{lemma:denting-points}
Let $X$, $Y$ be Banach spaces and let $G\in\mathcal{L}(X,Y)$ be a (norm-one) generating operator. If $x_0\in\dent(B_X)$, then $\|Gx_0\|=1$.
\end{lemma}

\begin{proof}
	Given $\delta>0$, observe that $x_0\in\overline{\att(G,\delta)}$. Otherwise, there would exist $r>0$ such that $B(x_0,r)\cap \att(G,\delta)=\emptyset$, so $\att(G,\delta)\subseteq B_X\setminus B(x_0,r)$ and
	$$
	x_0\in B_X=\overline{\conv}(\att(G,\delta))\subseteq \overline{\conv}( B_X\setminus B(x_0,r))
	$$
	which contradicts $x_0$ being a denting point of $B_X$. Consequently, $\|Gx_0\|\geq 1-\delta$ and the arbitrariness of $\delta$ finishes the proof.
\end{proof}

The above result can be slightly improved by using the following definition.

\begin{definition} \label{def-fragmenting-points} Let $x_0 \in S_X$. We say that $x_0$ is a point of \emph{sliced fragmentability} if for every $\delta > 0$ there is a  slice $S_\delta$ of $B_X$ such that $S_\delta \subset x_0 + \delta B_X$.
\end{definition}

Observe that this notion is weaker than that of denting point (for instance, points in the closure of the set of denting points are of sliced fragmentability but they do not need to be denting, even in the finite-dimensional case).

\begin{lemma}\label{lemma:fragmenting-points}
Let $X$, $Y$ be Banach spaces, let $G\in S_{\mathcal{L}(X,Y)}$ be a generating operator, and let $x_0 \in S_X$ be a point of sliced fragmentability, then $\|Gx_0\|=1$.
\end{lemma}

\begin{proof} Fixed $\delta>0$, by our assumption, $\overline{\conv} (\att(G, \delta)) = B_X$ for every $\delta >0$. This implies that, fixed $\delta>0$, the set $\att(G,\delta)$ intersects every slice of $B_X$. Applying this to the slice  $S_\delta$  from Definition~\ref{def-fragmenting-points}, we obtain that  there is a point $x_\delta \in S_\delta \cap \att(G,\delta)$. For this $x_\delta$, we have  $\|x_0 - x_\delta\| < \delta$ and $\|Gx_\delta\|> 1 - \delta$.  Consequently,
$$
\|Gx_0\| \geq \|Gx_\delta\| - \|G(x_0 - x_\delta)\|  \geq 1 - 2\delta
$$
and the arbitrariness of $\delta$ finishes the proof.
\end{proof}

We do not know if Lemma~\ref{lemma:fragmenting-points} is a characterization, but in Proposition~\ref{prop:charact-point-of-norm-attainment-Gen-op} we will characterize those points on which every generating operator attains its norm.

\begin{proposition}\label{prop:B_X-dentable}
Let $X$, $Y$ be Banach spaces and let $G\in\mathcal{L}(X,Y)$ be a norm-one operator. Suppose that $B_X=\overline{\conv}(\dent(B_X))$. Then, $G$ is generating if and only if $\|Gx\|=1$ for every $x\in \dent(B_X)$.
\end{proposition}

\begin{proof}
	If $\|Gx\|=1$ for every $x\in \dent(B_X)$, then  $\dent(B_X)\subset \att(G,\delta)$ for every $\delta>0$ and, therefore, $G$ is generating by Corollary~\ref{corollary:characterization-generating}.iv as $B_X=\overline{\conv}(\dent(B_X))$. The converse result follows from Lemma~\ref{lemma:denting-points}.
\end{proof}

\begin{corollary}\label{cor:domain-RNP}
Let $X$, $Y$ be Banach spaces and let $G\in\mathcal{L}(X,Y)$ be a norm-one operator. Suppose that $X$ has the Radon-Nikod\'{y}m property. Then, $G$ is generating if and only if $\|Gx\|=1$ for every $x\in \dent(B_X)$.
\end{corollary}

In the finite-dimensional case, the RNP is for free and denting points and extreme points coincide. Therefore, the following particular case holds.

\begin{corollary}\label{cor:X-finite-dimension}
Let $X$ be a finite-dimensional space, let $Y$ be a Banach space, and let $G\in\mathcal{L}(X,Y)$ be a norm-one operator. Then, $G$ is generating if and only if $\|Gx\|=1$ for every $x\in \ext(B_X)$.
\end{corollary}

The following particular case of Corollary~\ref{cor:domain-RNP} is especially interesting.

\begin{example}\label{exa:ell_1}
{\slshape Let $Y$ be a Banach space and let $G\in\mathcal{L}(\ell_1,Y)$ be a norm-one operator. Then, $G$ is generating if and only if $\|Ge_n\|=1$ for every $n\in\N$.}	
\end{example}

When every point of the unit sphere of the domain is a denting point, Proposition~\ref{prop:B_X-dentable} tells us that generating operators are isometric embeddings. Spaces with such property of the unit sphere are average locally uniformly rotund (ALUR for short) spaces. They were introduced in \cite{Troyanski} and it can be deduced from \cite[Theorem]{LinLinTroy} that a Banach space is ALUR if and only if every point of the unit sphere is a denting point.

\begin{corollary}\label{cor:domain-LUR}
Let $X$, $Y$ be Banach spaces and suppose that $X$ is ALUR. Then, every generating operator $G\in\mathcal{L}(X,Y)$ is an isometric embedding.
\end{corollary}

The next result gives another useful characterization of $r$-generating operators.

\begin{theorem}\label{thm:charact-rGenerating-spear-set}
Let $X$, $Y$ be Banach spaces, let $G\in\mathcal{L}(X,Y)$ be a norm-one operator, let $r\in(0,1]$, and let $\mathcal{A}\subset B_{Y^*}$ such that $\overline{\aconv}^{w^*}(\mathcal{A})=B_{Y^*}$. Then,
$G$ is $r$-generating if and only if $\displaystyle \max_{\theta\in\T}\sup_{y^*\in \mathcal{A}}\|G^*(y^*)+\theta x^*\|\geq1+r\|x^*\|$ for every $x^*\in X^*$.
\end{theorem}

\begin{proof}\parskip=0ex
If $G$ is $r$-generating, fixed $x^*\in X^*$ and $\delta>0$, we can write
\begin{align*}
\max_{\theta\in\T}\sup_{y^*\in \mathcal{A}} \|G^*(y^*)+\theta x^*\|
	 &=\max_{\theta\in\T}\sup_{y^*\in \mathcal{A}} \sup_{x\in B_X} |(G^*y^*)(x)+\theta x^*(x)| \\ &=\sup_{x\in B_X}\sup_{y^*\in \mathcal{A}}(|y^*(Gx)|+|x^*(x)|)
	\\ &=\sup_{x\in B_X}(\|Gx\|+|x^*(x)|) \geq \sup_{x\in\att(G,\delta)} (\|Gx\|+|x^*(x)|) \\ &\geq \sup_{x\in\att(G,\delta)}(1-\delta+|x^*(x)|) \geq 1-\delta+r\|x^*\|
\end{align*}
where the last inequality holds by Proposition~\ref{prop:characterization-rgenerating}. The arbitrariness of $\delta$ gives the desired inequality.

To prove the converse, fixed $x^*\in S_{X^*}$ and $\delta>0$, it suffices to show that $\|x^*\|_{G,\delta}\geq  r$ by Proposition~\ref{prop:characterization-rgenerating}. We use the hypothesis for $\frac{\delta}{2} x^*$ to get that
$$
\max_{\theta\in\T}\sup_{y^*\in \mathcal{A}} \left\|G^*(y^*)+\theta \frac{\delta}{2} x^*\right\|\geq 1+r\frac{\delta}{2}.
$$
So, given $0<\eps<\frac{\delta}{2}$, there are $y^*\in \mathcal{A}$, $\theta\in \T$, and $x\in B_X$ such that
$$
\|Gx\|+\frac{\delta}{2}|x^*(x)|\geq \left|y^*(Gx)+\theta \frac{\delta}{2} x^*(x)\right|>1+r\frac{\delta}{2}-\eps
$$
which implies that $$\frac{\delta}{2}|x^*(x)|>r\frac{\delta}{2}-\eps\quad \text{ and } \quad \|Gx\|>1+(r-1)\frac{\delta}{2}-\eps\geq 1-\delta.$$
The arbitrariness of $\eps$ gives $\|x^*\|_{G,\delta}\geq r$ as desired.
\end{proof}

Of course, one can always use $\mathcal{A}=B_{Y^*}$ in Theorem~\ref{thm:charact-rGenerating-spear-set} if no other interesting choice for $\mathcal{A}$ is available and still one obtains a useful characterization of $r$-generating operators.

In the case of generating operators, we emphasize the following result.

\begin{corollary}\label{corollary:charact-generating-dual-spear}
Let $X$, $Y$ be Banach spaces, let $\mathcal{A}\subset B_{Y^*}$ be one-norming for $Y$, and let $G\in \mathcal{L}(X,Y)$ with $\|G\|=1$. Then, the following are equivalent:
\begin{enumerate}[$(i)$]
\item $G$ is generating.
\item $G^*(B_{Y^*})$ is a spear set of $X^*$.
\item $G^*(\mathcal{A})$ is a spear set of $X^*$.
\item $\displaystyle \max_{\theta\in \T}\sup_{y^*\in B_{Y^*}}\|G^*(y^*) + \theta x^*\|=2$ for every $x^*\in S_{X^*}$.
\end{enumerate}
\end{corollary}

Only item $(iv)$ is new, and follows immediately from the following remark.

\begin{remark}\label{rem:spearset-only-for-SX}
{\slshape Let $Z$ be a Banach space and $F\subset B_Z$. Then, $F$ is a spear set if and only if
$\max\limits_{\theta\in \T}\sup\limits_{z\in F}\|z + \theta z_0\|=2$ for every $z_0\in S_Z$.}\newline \indent
Indeed, to prove the sufficiency, fixed $0\neq z_1\in X$, observe that
$$
\max\limits_{\theta\in \T}\sup\limits_{z\in F}\left\|z+\theta  \frac{z_1}{\|z_1\|}\right\|=2
$$
implies that $\max\limits_{\theta\in \T}\sup\limits_{z\in F}\bigl\|\|z_1\|z+\theta z_1\bigr\|=2\|z_1\|$. So, if $\|z_1\|\geq1$, the triangle inequality allows to write
$$
\max\limits_{\theta\in \T}\sup\limits_{z\in F}\left\|z+\theta  z_1\right\|\geq \max\limits_{\theta\in \T}\sup\limits_{z\in F}\bigl\|\|z_1\|z+\theta  z_1\bigr\|-(\|z_1\|-1)=1+\|z_1\|.
$$
If otherwise $\|z_1\|<1$, just observe that
\[
\pushQED{\qed}
\max\limits_{\theta\in \T}\sup\limits_{z\in F}\left\|z+\theta  z_1\right\|\geq \max\limits_{\theta\in \T}\sup\limits_{z\in F}\left\|z+\theta  \frac{z_1}{\|z_1\|}\right\|-(1-\|z_1\|)=1+\|z_1\|.\qedhere\popQED
\]
\end{remark}

What we have shown is that it suffices to use elements $x^*\in S_{X^*}$ in Theorem~\ref{thm:charact-rGenerating-spear-set} when $r=1$. However, the following example shows that this is not the case for any other value of $0<r<1$.

\begin{example}
{\slshape Let $0<r<1$ be fixed, let $X$ be the real two-dimensional Hilbert space, $\{e_1, e_2\}$ be its orthonormal basis with $\{e_1^*, e_2^*\}$ being the corresponding coordinate functionals. The norm-one operator $G\in \mathcal{L}(X)$ given by
$G=r\Id+(1-r)e_1^*\otimes e_1$
is not $r$-generating but satisfies
$$
\max_{\theta\in\T}\sup_{x^*\in B_{X^*}}\|G^*(x^*)+\theta x^*\|\geq 1+r\|x^*\|$$ for every $x^*\in S_{X^*}$.}\newline
\indent Indeed, it is clear that $\|G\|=1$ and $G^*=r\Id+(1-r)e_1\otimes e_1^*$. So, given $x^*\in S_{X^*}$, we have that
\begin{align*}
\max_{\theta\in\T}\sup_{x^*\in B_{X^*}}\|G^*(x^*)+\theta x^*\|&\geq \|G^*(x^*)+x^*\|=\|(1+r)x^*+(1-r)x^*(e_1)e_1^*\|\\ &=\|2x^*(e_1)e_1^*+(1+r)x^*(e_2)e_2^*\|\geq 1+r.
\end{align*}
Observe that $G$ attains its norm only at $\pm e_1$ so Proposition~\ref{prop:characterization-rGenerating-Xfindim} tells us that $G$ is not $r$-generating (in fact, it is not $s$-generating for any $0<s\leq 1$).\qed
\end{example}

If we are able to guarantee that $G^*(B_{Y^*})$ is a spear set of $X^*$, Corollary~\ref{corollary:charact-generating-dual-spear} shows that $G$ is generating. The most naive way to do so is to require $G^*(B_{Y^*})=B_{X^*}$ but observe that, as $\|G^*\|=1$, this implies that $G^*$ is surjective and $G$ is an isometry.

The other extreme possibility is $G^*(B_{Y^*})=\{\lambda x_0^*\colon \lambda\in \K,\, |\lambda|\leq 1\}$ for some $x_0^*\in S_{X^*}$. This obviously means that $G$ is a rank one operator; in this case, $G^*(B_{Y^*})$ is a spear set of $X^*$ if and only if $x_0^*$ is a spear vector of $X^*$. In this particular case, Corollary~\ref{corollary:charact-generating-dual-spear} reads as follows.

\begin{corollary}\label{cor:charact-generatin-rank1}
Let $X$, $Y$ be Banach spaces, $x_0^*\in S_{X^*}$, and $y_0\in S_Y$. Then, the rank-one operator $G=x_0^* \otimes y_0$ is generating if and only if $x_0^*\in \Spear(X^*)$.
\end{corollary}

Observe the similarity with \cite[Corollary~5.9]{KMMP-Spear} which states that $G=x_0^* \otimes y_0$ is spear if and only if $x_0^*$ is a spear functional and $y_0$ is a spear vector. Here the condition is easier to satisfy, of course.

\subsection{Some stability results}
The following result shows that the property of being generating is stable by $c_0$-, $\ell_1$-, and $\ell_\infty$-sums of Banach spaces.

\begin{proposition}\label{prop:stability-sums}
Let $\{X_\lambda\colon \lambda\in\Lambda\}$, $\{Y_\lambda\colon \lambda\in\Lambda\}$ be two families of Banach spaces and let $G_\lambda\in\mathcal{L}(X_\lambda,Y_\lambda)$ be a norm-one operator for every $\lambda\in\Lambda$. Let $E$ be one of the Banach spaces $c_0$, $\ell_\infty$, or $\ell_1$, let $X=\left[\bigoplus_{\lambda\in\Lambda} X_\lambda\right]_E$ and $Y=\left[\bigoplus_{\lambda\in\Lambda} Y_\lambda\right]_E$, and define the operator $G\colon X\longrightarrow Y$ by
$$ G\left[(x_\lambda)_{\lambda\in\Lambda}\right]=(G_\lambda x_\lambda)_{\lambda\in\Lambda} $$
for every $(x_\lambda)_{\lambda\in\Lambda}\in \left[\bigoplus_{\lambda\in\Lambda} X_\lambda\right]_E$. Then, $G$ is generating if and only if $G_\lambda$ is generating for every $\lambda\in\Lambda$.
\end{proposition}

\begin{proof}\parskip=0ex
	Suppose first that $G$ is generating and, fixed $\kappa\in\Lambda$, let us show that $G_\kappa$ is generating. Observe that calling $W=\left[\bigoplus_{\lambda\neq\kappa}
	X_\lambda\right]_E$ and $Z=\left[\bigoplus_{\lambda\neq\kappa}
	Y_\lambda\right]_E$, we can write $X=X_\kappa\oplus_\infty W$ and $Y=Y_\kappa\oplus_\infty Z$ when $E$ is $\ell_\infty$ or $c_0$ and $X=X_\kappa\oplus_1 W$ and $Y=Y_\kappa\oplus_1 Z$ when $E$ is $\ell_1$. Given $T_\kappa\in \mathcal{L}(X_\kappa,Y_\kappa)$, define $T\in \mathcal{L}(X,Y)$ by
	$$T(x_\kappa,w)=(T_\kappa x_\kappa,0) \qquad (x_\kappa\in X_\kappa,\ w\in W)$$
	which satisfies $\|T\|=\|T_\kappa\|$ and $\|T\|_G=\|T\|$ as $G$ is generating. Moreover,
	$$ \|T\|_G=\inf_{\delta>0}\sup\{\|T(x_\kappa,w)\|\colon (x_\kappa,w)\in\att(G,\delta)\}=\inf_{\delta>0}\sup\{\|T_\kappa x_\kappa\|\colon x_\kappa\in\att(G_\kappa,\delta)\}=\|T_\kappa\|_{G_\kappa}, $$
	thus $\|T_\kappa\|=\|T_\kappa\|_{G_\kappa}$. The arbitrariness of $T_\kappa$ gives that $G_\kappa$ is generating.
	
	To prove the sufficiency when $E$ is $c_0$ or $\ell_{\infty}$, given $T\in\mathcal{L}(X,Y)$, it is enough to show that $\|T\|_G\geq \|T\|$. Fixed $\eps>0$, we may find $\kappa\in\Lambda$ such that $\|P_\kappa T\|>\|T\|-\eps$, where $P_\kappa$ denotes the projection from $Y$ onto $Y_\kappa$. Now, writing $X= X_\kappa\oplus_\infty W$ where $W=\left[\bigoplus_{\lambda\neq\kappa}
	X_\lambda\right]_E$, we have that $B_X=\conv\left(S_{X_\kappa}\times S_W\right)$ and so we may find $x_0\in S_{X_\kappa}$ and $w_0\in S_W$ such that
	$$
	\|P_\kappa T(x_0,w_0)\|>\|T\|-\eps.
	$$
	Take $x_0^*\in S_{{X_\kappa}^*}$ with $x_0^*(x_0)=1$ and define the operator $S\in\mathcal{L}(X_\kappa,Y_\kappa)$ by
	$$
	S(x)=P_\kappa T(x,x_0^*(x)w_0) \qquad (x\in X_\kappa)
	$$
	which satisfies $\|S\|\geq \|Sx_0\|=\|P_\kappa T(x_0,w_0)>\|T\|-\eps$ and $\|S\|_{G_\kappa}=\|S\|$ since $G_\kappa$ is generating. Moreover, fixed $\delta>0$,
	\begin{align*}
		\|T\|_{G,\delta} &=\sup\{\|Tx\|\colon x\in S_X, \, \|Gx\|>1-\delta\} \\
		&\geq \sup\{\|T(x,x_0^*(x)w_0)\| \colon x\in X_\kappa, \, (x,x_0^*(x)w_0)\in S_X, \, \|G(x,x_0^*(x)w_0)\|>1-\delta \} \\
		&\geq \sup\{\|P_\kappa T(x,x_0^*(x)w_0)\| \colon x\in S_{X_\kappa}, \, \|G_\kappa x\|>1-\delta \} =\|S\|_{G_\kappa,\delta}.
	\end{align*}
	Therefore, $\|T\|_G\geq\|S\|_{G_\kappa}=\|S\|>\|T\|-\eps$ and the arbitrariness of $\eps$ gives that $\|T\|_G\geq \|T\|$ as desired.
	
	In the case when $E=\ell_1$, fixed $\delta>0$, consider the set
	$$A_\delta:=\bigcup_{\lambda\in\Lambda} \left\{x\in X\colon x_\lambda\in\att(G_\lambda,\delta), \, x_\kappa=0 \textnormal{ if } \kappa\neq\lambda\right\},$$
	which satisfies that $A_\delta\subseteq \att(G,\delta)$ and
	\begin{align*}
		\overline{\conv}(A_\delta)&\supseteq\bigcup_{\lambda\in\Lambda} \overline{\conv}\left(\left\{x\in X\colon x_\lambda\in\att(G_\lambda,\delta), \, x_\kappa=0 \textnormal{ if } \kappa\neq\lambda\right\}\right) \\
		&=\bigcup_{\lambda\in\Lambda} \left\{x\in X\colon x_\lambda\in B_{X_\lambda}, \, x_\kappa=0 \textnormal{ if } \kappa\neq\lambda\right\},
	\end{align*}
	where in the last equality we have used Corollary~\ref{corollary:characterization-generating}.iv as $G_\lambda$ is generating for every $\lambda\in\Lambda$. Therefore, $B_X=\overline{\conv}\left(\bigcup_{\lambda\in\Lambda} \left\{x\in X\colon x_\lambda\in B_{X_\lambda}, \, x_\kappa=0 \textnormal{ if } \kappa\neq\lambda\right\}\right)\subseteq\overline{\conv}(A_\delta)\subseteq\overline{\conv}(\att(G,\delta))$ and the arbitrariness of $\delta$ gives that $G$ is generating by Corollary~\ref{corollary:characterization-generating}.iv.
\end{proof}

We next discuss the relationship of being generating with the operation of taking the adjoint.
We show next that if the second adjoint is $r$-generating then the operator itself is $r$-generating.

\begin{proposition}\label{prop:G**-r-generating}
Let $X$, $Y$ be Banach spaces, let $G\in\mathcal{L}(X,Y)$ be a norm-one operator, and let $r\in(0,1]$. If $G^{**}$ is $r$-generating, then $G$ is also $r$-generating.
\end{proposition}

\begin{proof} Fixed $x_0^*\in S_{X^*}\subset X^{***}$, we have that $\|x_0^*\|_{G^{**},\delta}\geq r \|x^*\|$ for every $\delta>0$ by Proposition~\ref{prop:characterization-rgenerating}. So, fixed $\delta>0$ and $\eps>0$, there exists $x^{**}\in \att(G^{**},\delta)$ with $|x^{**}(x_0^*)|>(1-\eps)r$. Now, as $\|G^{**}x^{**}\|>1-\delta$, there is $y^*\in S_{X^*}$ satisfying $|x^{**}(G^*y^*)|>1-\delta$. By Goldstine's theorem there is $x\in B_X$ such that
$$
|x_0^*(x)|=|J_X(x)(x_0^*)|>(1-\eps)r \qquad \textnormal{and} \qquad \|Gx\|\geq |y^*(Gx)|=|J_X(x)(G^*y^*)|>1-\delta
$$
which gives $\|x_0^*\|_{G,\delta}\geq r$ since $\eps>0$ was arbitrary. So $G$ is $r$-generating by Proposition~\ref{prop:characterization-rgenerating}.
\end{proof}

We do not know if the converse of the above result holds in general or even for $r=1$. On the other hand, the following example shows that there is no good behaviour of the property of being generating with respect to taking one adjoint, as the property does not pass from an operator to its adjoint, nor the other way around.

\begin{example}
{\slshape 	Consider the norm-one operator $G\colon c_0\longrightarrow c_0$ defined by
	$$
	Gx=\sum_{n=1}^{\infty} \dfrac{1}{n} x(n) e_n \qquad (x\in c_0).
	$$
	For any $x\in S_{c_0}$ with $x(1)\in \T$ we have that $\|G(x)\|=1$ and, consequently, $x\in\att(G,\delta)$ for every $\delta>0$. Since such elements are enough to recover the whole unit ball of $c_0$ by taking closed convex hull, $G$ is generating by Corollary~\ref{corollary:characterization-generating}.iv.}
\begin{itemize}
  \item \slshape The adjoint operator $G^*\colon\ell_1\longrightarrow\ell_1$
	$$
	G^*(x^*)=\sum_{n=1}^{\infty} \dfrac{1}{n} x^*(n) e_n^* \qquad (x^*\in\ell_1)
	$$
	is not generating by Example~\ref{exa:ell_1} since $\|G^*(e_n^*)\|=\frac{1}{n}<1$ for $n>1$.
  \item \slshape The second adjoint $G^{**}\colon\ell_\infty\longrightarrow\ell_\infty$
	$$
	G^{**}(x^{**})=\sum_{n=1}^{\infty} \dfrac{1}{n} x^{**}(n) e_n^{**} \qquad (x^{**}\in \ell_\infty)
	$$
	is again generating following an analogous argument to the one used for $G$, using this time elements $x\in S_{\ell_\infty}$ with $x(1)\in \T$.
\end{itemize}	
\end{example}

\subsection{Some examples in classical Banach spaces}\label{subsect:classical}
Our aim here is to provide some characterizations of generating operators when the domain space is $L_1(\mu)$ or the range space is $C_0(L)$ by making use of Corollary~\ref{corollary:charact-generating-dual-spear}.

\subsubsection{Operators acting from $L_1(\mu)$}
Let $Y$ be a Banach space and let $(\Omega,\Sigma,\mu)$ be a finite measure space. Recall that an operator $T\in\mathcal{L}(L_1(\mu),Y)$ is \emph{representable} if there exists $g\in L_\infty(\mu,Y)$ such that
$$
T(f)=\int_{\Omega} f(t)g(t)\,d\mu(t) \qquad (f\in L_1(\mu)).
$$
In such case, $\|T\|=\|g\|_\infty$. Moreover, its adjoint $T^*\colon Y^*\to L_\infty(\mu)$ is given by
$$
[T^*(y^*)](f)=y^*(T(f))=\int_{\Omega} f(t)y^*(g(t))\,d\mu(t) \qquad (f\in L_1(\mu), \ y^*\in Y^*),
$$
then $T^*(y)=y^*\circ g \in L_\infty(\mu)$ for $y^*\in Y^*$.

Weakly compact operators are representable (see \cite[p.~65, Theorem~12]{DiestelUhl}, for instance). If $Y$ has the RNP, then every operator in $\mathcal{L}(L_1(\mu),Y)$ is representable (see \cite[p.~63, Theorem~5]{DiestelUhl}, for instance) and so $\mathcal{L}(L_1(\mu),Y)$ identifies with $L_\infty(\mu, Y)$ in this case.

The question of which operators acting from $L_1(\mu)$ are generating leads to study the spear sets in $L_\infty(\mu)$. We do so in the next result which is valid for arbitrary measures.

\begin{proposition}[Spear sets in $B_{L_\infty(\mu)}$]
Let $(\Omega,\Sigma,\mu)$ be a positive measure space and let $F\subset B_{L_\infty(\mu)}$. Then, the following are equivalent:
\begin{enumerate}[$(i)$]
	\item  $F$ is a spear set.
	\item  For every measurable set $A\in\Sigma$ with $\mu(A)\neq0$ and every $\eps>0$ there exists $B\in\Sigma$, $B\subset A$ with $\mu(B)\neq0$ and $f\in F$ such that $|f(t)|>1-\eps$ for every $t\in B$.
\end{enumerate}
\end{proposition}

\begin{proof}
Suppose first that $F$ is a spear set. Given $A\in\Sigma$ with $\mu(A)\neq0$ and $\eps>0$, since
$$
\max_{\theta \in \T}\sup_{f\in F}\left\|f+\theta\mathbbm{1}_{A}\right\|_\infty=2,
$$
there exists $f_0\in F$ and $\theta_0\in\T$ such that $\|f_0+\theta_0\mathbbm{1}_A\|_\infty>2-\eps$ and thus, there exists $B\subset A$ with $\mu(B)\neq0$ such that $|f(t)|>1-\eps$ for every $t\in B$. To prove the converse implication, given $x\in L_\infty(\mu)$ and $\eps>0$, there is $A\in\Sigma$ with $\mu(A)\neq0$ such that $|x(t)|\geq\|x\|_\infty-\eps$ for every $t\in A$. By the hypothesis, there is a subset $B$ of $A$ with $\mu(B)\neq0$ and $f\in F$ such that $|f_0(t)|>1-\eps$ for every $t\in B$. Now, thanks to the compactness of $\T$ we can fix an $\eps$-net $\T_\eps$ of $\T$, then we may find $\theta_0 \in \T_\eps$ and $C\subset B$ with $\mu(C)\neq0$ such that $|f_0(t)+\theta_0 x(t)|\geq |f_0(t)|+|x(t)|(1-\eps)$ for every $t\in C$. Therefore,
\begin{align*}
	\max_{\theta \in \T}\sup_{f\in F}\|f+\theta x\|_\infty &\geq \max_{\theta \in \T}\|f_0+\theta x\|_\infty \geq \inf_{t\in C} |f_0(t)+\theta x(t)| \geq \inf_{t\in C} |f_0(t)|+|x(t)|(1-\eps)\\
	&\geq 1-\eps+(\|x\|_\infty-\eps)(1-\eps),
\end{align*}
and the arbitrariness of $\eps$ gives $ \max_{\theta \in \T}\sup_{f\in F}\|f+\theta x\|_\infty \geq 1+\|x\|_\infty$.
\end{proof}

As an immediate consequence we get the following characterization of generating representable operators acting on $L_1(\mu)$.

\begin{corollary}\label{cor:char-generating-from-L1}
Let $Y$ be a Banach space, let $(\Omega,\Sigma,\mu)$ be a finite measure space, and let $G\in\mathcal{L}(L_1(\mu),Y)$ be a norm-one operator which is representable by $g\in L_\infty(\mu, Y)$. Then, the following are equivalent:
\begin{enumerate}[$(i)$]
\item $G$ is generating.
\item $\{y^*\circ g\colon y^*\in B_{Y^*} \}$ is a spear set of $B_{L_\infty(\mu)}$.
\item For every measurable set $A\subset \Omega$ with $\mu(A)>0$ and every $\eps>0$ there exists $B\subset A$ with $\mu(B)>0$ such that $\|g(t)\|>1-\eps$ for all $t\in B$.
\item $\|g(t)\|=1$ $\mu$-almost everywhere.
\end{enumerate}
\end{corollary}

\begin{remark}\label{remark:finite-sigma-finite}
{\slshape  The restriction on the measure $\mu$ being finite in Corollary~\ref{cor:char-generating-from-L1} can be relaxed to being $\sigma$-finite}.\newline
\indent Indeed,  given a $\sigma$-finite measure $\mu$, there is a suitable probability measure $\nu$ such that $L_1(\mu)\equiv L_1(\nu)$ and $L_\infty(\mu,Y)\equiv L_\infty(\nu,Y)$, see \cite[Proposition~1.6.1]{CembranosMendoza} for instance).
\end{remark}

Compare Corollary~\ref{cor:char-generating-from-L1} the above result with \cite[Corollary~4.22]{KMMP-Spear} which says that $G\in\mathcal{L}(L_1(\mu),Y)$ of norm-one which is representable by $g\in L_\infty(\mu,Y)$ is a spear operator if and only if it has the alternative Daugavet property if and only if $g(t)\in \Spear(Y)$ for a.e.\ $t\in \Omega$. It is then easy to construct generating operators from $L_1(\mu)$ which do not have the alternative Daugavet property: for instance, $G\in \mathcal{L}(L_1[0,1],\ell_2^2)$ given by $G(f)=\displaystyle \int_0^1 f(t)(\cos(2\pi t),\sin(2\pi t))\, dt$ for every $f\in L_1[0,1]$.

\subsubsection{Operators arriving to $C_0(L)$}
Let $L$ be a Hausdorff locally compact topological space. It is immediate from the definition of the norm, that the set $\mathcal{A}=\{\delta_t\colon t\in L\}\subset C_0(L)^*$ is one-norming for $C_0(L)$. Hence, Corollary~\ref{corollary:charact-generating-dual-spear} reads in this case as follows.

\begin{proposition}
Let $X$ be a Banach space, let $L$ be a Hausdorff locally compact topological space, and let $G\in \mathcal{L}(X,C_0(L))$ be a norm-one operator. Then, the following are equivalent:
\begin{enumerate}[$(i)$]
\item $G$ is generating.
\item The set $\{G^*(\delta_t)\colon t\in L\}$ is a spear set of $X^*$.
\end{enumerate}
\end{proposition}

We would like to compare the result above with \cite[Proposition~4.2]{KMMP-Spear} where it is proved that $G\in \mathcal{L}(X,C_0(L))$ has the alternative Daugavet property if and only if $\{G^*(\delta_t)\colon t\in U\}$ is a spear set of $X^*$ for every open subset $U\subset L$. It is then easy to construct examples of generating operators arriving to $C_0(L)$ spaces which do not have the alternative Daugavet property. For instance, consider $G\in\mathcal{L}(c_0,c_0)$ given by $$
[Gx](n)=\begin{cases} 0 & \text{if $n$ is odd,}\\ x(n) & \text{if $n$ is even.}\end{cases}
$$

\section{Generating operators and norm-attainment}\label{Sect-normattaiment}
We discuss here when generating operators are norm-attaining. On the one hand, it is shown in \cite[Theorem 2.9]{KMMP-Spear} that every spear $x^*\in X^*$ attains its norm. So rank-one generating operators also attain their norm by Corollary~\ref{cor:charact-generatin-rank1}.

\begin{corollary}\label{cor:x*-generating-attains-norm}
Let $X$, $Y$ be Banach spaces and $G\in \Gen(X,Y)$ of rank-one. Then, $G$ attains its norm.
\end{corollary}

Besides, if $B_X$ contains denting points, all generating operators with domain $X$ are norm attaining by Lemma~\ref{lemma:denting-points}.

On the other hand, operators with the alternative Daugavet property are generating (see Example~\ref{examples:preliminaryones}.(4)), and there are operators with the alternative Daugavet property which do not attain their norm (see \cite[Example 8.7]{KMMP-Spear}). The construction of the cited example in \cite{KMMP-Spear} is not easy at all, but we may construct easier examples of generating operators which do not attain their norm, even with rank two.

\begin{example}\label{example:generating-rank-two-NONA}
{\slshape Consider $g\colon [0,1]\longrightarrow \ell_2^2$ given by $g(t)=(\cos t,\sin t)$ and the norm-one operator $G\in\mathcal{L}(L_1[0,1],\ell_2^2)$ represented by $g$:
$$
G(x) = \int_0^1 x(t)g(t)\, dt \qquad (x\in L_1[0,1]).
$$
Then, $G$ is generating but does not attain its norm.}
\end{example}

\begin{proof}
Observe that $G$ is generating by Corollary~\ref{cor:char-generating-from-L1} as $\|g(t)\|=1$ for every $t\in [0,1]$. To prove that $G$ does not attain its norm, recall that for an integrable complex-valued function $f$ the equality $\left|\int_0^1 f(t)dt \right|=\int_0^1|f(t)|dt$ holds if and only if there is $\lambda\in \T$ such that $f=\lambda |f|$ except for a set of zero measure. Suppose, to find a contradiction, that there is a non-zero $x\in L_1[0,1]$ satisfying $\|Gx\|=\|x\|$. Then, as $xg$ can be seen as a complex-valued function and we can identify the norm on $\ell_2^2$ with the modulus in $\C$, we have that
\begin{align*}
\left|\int_0^1 x(t)g(t)dt\right|&=\left\|\int_0^1x(t)g(t)dt\right\|\\
&=\|Gx\|=\|x\|=\int_0^1|x(t)|dt=\int_0^1|x(t)g(t)|dt.
\end{align*}
Therefore, there is $\lambda\in \T$ such that $xg=\lambda |xg|=\lambda |x|$ except for a set of zero measure. But this is impossible since $x$ takes real values and $g$ covers a non-trivial arc of the unit circumference.
\end{proof}

Example~\ref{example:generating-rank-two-NONA} can be generalized for other two-dimensional spaces $Y$, but we need some assumptions on the shape of $S_Y$. If $S_Y$ can be expressed as a finite or countable union of segments, then every generating operator $G\in\mathcal{L}(L_1[0,1],Y)$ attains its norm, leading to a complete characterization.

\begin{proposition}
Let $Y$ be a real two-dimensional space. Then, the following are equivalent:
\begin{enumerate}[$(i)$]
	\item $S_Y$ is a finite or countable union of segments.
	\item Every generating operator $G\in\mathcal{L}(L_1[0,1],Y)$ attains its norm.
\end{enumerate}
Moreover, if the previous assertions hold, we have that $B_{L_1[0,1]}=\overline{\conv}(\att(G))$ for every generating operator $G\in \mathcal{L}(L_1[0,1],Y))$.
\end{proposition}
\begin{proof}
$(i)\Rightarrow(ii)$ Let $G\in\mathcal{L}(L_1[0,1],Y)$ be a generating operator. Since $Y$ has dimension two, $G$ can be represented by
	$$
	G(x)=\int_0^1 x(t)g(t) \,dt \qquad (x\in L_1[0,1])
	$$
	for a suitable $g\in L_\infty([0,1],Y)$ with $\|g\|_\infty=1$ and $\|g(t)\|=1$ almost everywhere by Corollary~\ref{cor:char-generating-from-L1}. Since $S_Y$ is a finite or countable union of segments, we may find a partition $\pi$ of $[0,1]$  in measurable subsets of positive measure such that $g(A)$ is contained in a segment of $S_Y$ almost everywhere for every $A\in \pi$. Then, for every $\Delta \in \pi$ and every measurable subset $A \subset \Delta$ of positive measure, consider $x_A=\frac{1}{|A|} \mathbbm{1}_{A}\in S_{L_1[0,1]}$, where $|A|$ denotes the Lebesgue measure of $A$, and let us show that $G$ attains its norm at $x_A$. Indeed, as $g(A)$ is contained in a segment of $S_Y$ a.e., there exists $y^*\in S_{Y^*}$ such that $y^*(g(t))=1$ a.e.\ in $A$, thus
	$$
	\|G(x_A)\|\geq y^*(Gx_A)= y^*\left( \int_0^1 \dfrac{1}{|A|}\mathbbm{1}_{A}(t) g(t) \, dt\right)= \dfrac{1}{|A|} \int_A y^*(g(t)) \, dt =1,
	$$
	and so $\|G(x_A)\|=1$ as desired.\\
	Moreover, for this $\pi$
	$$
	B_{L_1[0,1]}\subseteq{\overline{\aconv}\left(\left\{\frac{1}{|A|} \mathbbm{1}_{A} \colon A \subset \Delta, \Delta \in\pi,  |A| > 0\right\}\right)}\subseteq\overline{\conv}(\att(G)),
	$$
	hence $B_{L_1[0,1]}=\overline{\conv}(\att(G))$.
	
	To prove $(ii)\Rightarrow(i)$, suppose that $S_Y$ cannot be written as a finite or countable union of segments and let us construct a generating operator $G\in\mathcal{L}(L_1[0,1],Y)$ not attaining its norm. Observe that the number of open maximal segments in $S_Y$ is finite or countable as $S_Y$ is a curve on a two-dimensional space with finite length. Let $\Delta_n$, $n\in \N$, be the open maximal segments in $S_Y$ and denote $D=S_Y\setminus \left(\cup_{n\in\N} \Delta_n\right)$. Clearly, $D$ is an uncountable metric compact subset of $S_Y$, hence it contains a homeomorphic copy of the Cantor set $K$ \cite[Chapter~I]{Kechris} and so there exists an injective continuous function $\varphi\colon K\longrightarrow D$. Now, let us construct an injection from $[0,1]$ to $K$. To do so, recall that the Cantor set is the set of numbers of $[0,1]$ that have a triadic representation consisting purely of $0$'s and $2$'s, that is,
	$$
	K=\left\{y\in[0,1] \colon y=\sum_{k=1}^{\infty} \dfrac{\beta_k}{3^k}, \ \beta_k=0,2 \right\}.
	$$
	Every $t\in[0,1]$ has a dyadic representation:
	$$
	t=\sum_{k=1}^\infty \dfrac{\alpha_k(t)}{2^k},
	$$
	where $\alpha_k(t)\in\{0,1\}$. This representation is unique except for a countable subset of $[0,1]$ consisting of those numbers with finite dyadic representation. Consider $\phi\colon [0,1]\longrightarrow K$ given by
	$$
	\phi(t)=\sum_{k=1}^\infty \dfrac{2\alpha_k(t)}{3^k} \qquad (t\in[0,1]),
	$$
	where $\alpha_k(t)\in\{0,1\}$ are the coefficients in the dyadic representation of $t$. The function $\phi$ is well-defined almost everywhere on $[0,1]$, injective, measurable, and its image lies on $K$. Then, the function $g=\varphi \circ \phi\colon [0,1]\longrightarrow D$ is well-defined almost everywhere on $[0,1]$, $g\in L_\infty[0,1]$, and it is injective. Consider the operator $G\colon L_1[0,1]\longrightarrow Y$ defined by
	$$
	G(x)=\int_0^1 x(t)g(t) \, dt \qquad (x\in L_1[0,1]).
	$$
	$G$ is generating by Corollary~\ref{cor:char-generating-from-L1} as $\|g(t)\|=1$ almost everywhere but it does not attain its norm. Indeed, suppose on the contrary that there is a non-zero $x\in L_1[0,1]$ such that
	$$
	\|G(x)\|=\left\|\int_0^1 x(t)g(t) \, dt\right\|=\int_0^1 |x(t)| \, dt = \|x\|.
	$$
	We may find $y_0^*\in S_{Y^*}$ such that
	$$\int_0^1 |x(t)|\, dt=\left\|\int_0^1 x(t)g(t)\,dt\right\|= y_0^*\left(\int_0^1 x(t)g(t)\, dt\right)=\int_0^1 x(t)y_0^*(g(t))\,dt.$$
	This equality implies the existence of a measurable subset $A$ of $[0,1]$ with positive measure such that $|x(t)|=x(t)y_0^*(g(t))$ for every $t\in A$, thus $y_0^*(g(t))\in\{1,-1\}$ for every $t\in A$. Note that $g(A)\subseteq \big\{y\in D\colon y_0^*(y)\in\{1,-1\}\big\}$. However, this leads to a contradiction. On the one hand, the latter set has at most four elements as $D$ does not contain open segments of $S_Y$. On the other hand, since $g$ is injective and $A$ has positive measure, $g(A)$ has infinitely many elements. Thus, $G$ cannot attain its norm.
\end{proof}

The next example shows that, even in the case of norm-attaining operators, the set $\att(G)$ cannot be used to characterize when $G$ is generating outside the case when $X$ is reflexive and $G$ is compact covered by Proposition~\ref{prop:X-reflexive-G-compact}.

\begin{example}\label{example:generatingNonNA-convattemptyinterior}
{\slshape Let $G\in\mathcal{L}(X,Y)$ be a generating operator between two Banach spaces $X$ and $Y$ such that it does not attain its norm. Then, the operator $\widetilde{G}\colon X\oplus_1 \K\longrightarrow Y\oplus_1 \K$ defined by $\widetilde{G}(x,\lambda)=(Gx,\lambda)$ is generating by Proposition~\ref{prop:stability-sums} and attains its norm, but $\overline{\conv}(\att(\widetilde{G}))=\overline{\conv}(\{(0,\lambda)\colon \lambda\in \T\})=\{(0,\lambda)\colon \lambda\in B_{\K}\}$ does not contain any ball of $X\oplus_1\K$.}
\end{example}

The following result characterizes the possibility to construct a generating operator not attaining its norm acting from a given Banach space which somehow extend Example~\ref{example:generating-rank-two-NONA}.

\begin{theorem}\label{theorem:char-when-X-domain-generating-NONA}
Let $X$ be a Banach space, the following are equivalent:
\begin{enumerate}[$(i)$]
\item There exists a Banach space $Y$ and a norm-one operator $G\in\mathcal{L}(X,Y)$ such that $G$ is generating but $\att(G)=\emptyset$.
\item There exists a spear set $\mathcal{B}\subseteq B_{X^*}$ such that $\sup\limits_{x^*\in \mathcal{B}} |x^*(x)|<1$ for every $x\in S_X$.
\end{enumerate}
\end{theorem}

\begin{proof}\parskip=0ex
	$(i)\Rightarrow(ii)$ Taking $\mathcal{B}=G^*(B_{Y^*})$, since $G$ is generating, we can use Corollary~\ref{corollary:charact-generating-dual-spear} to deduce that $\mathcal{B}$ is a spear set. Besides, as $G$ does not attain its norm, we have that
	$$
	1>\|G(x)\|=\sup_{y^*\in B_{Y^*}} |y^*(Gx)|=\sup_{y^*\in B_{Y^*}} |(G^*y^*)(x)|
	=\sup_{x^*\in \mathcal{B}} |x^*(x)|
	$$
	for every $x\in S_X$.
	
	$(ii)\Rightarrow(i)$  Consider $Y=\ell_\infty(\mathcal{B})$ and $G\colon X\longrightarrow\ell_\infty(\mathcal{B})$ defined by
	$$
	(Gx)(x^*)=x^*(x) \qquad (x^*\in X^*,\, x\in X).
	$$
	On the one hand, for $x\in S_X$, we have that
	$$
	\|G(x)\|=\sup_{x^*\in \mathcal{B}} |(Gx)(x^*)|=\sup_{x^*\in \mathcal{B}} |x^*(x)|<1.
	$$
	On the other hand, using that $\mathcal{B}$ is a spear set, for every $\eps>0$ we may find $x^*\in \mathcal{B}$ with $\|x^*\|>1-\eps$ and so
	$$
	\|G\|=\sup_{x\in B_X} \|G(x)\|\geq \sup_{x\in B_X} |(Gx)(x^*)|= \sup_{x\in B_X} |x^*(x)|=\|x^*\|>1-\eps.
	$$
	Therefore, $\|G\|=1$ but the norm is not attained.
	
	To show that $G$ is generating, we start claiming that, for every $g\in \ell_1(\mathcal{B})\subset \ell_\infty(\mathcal{B})^*$, we have
	$$
	G^*(g)=\sum_{x^*\in \mathcal{B}} g(x^*)x^*\in X^*.
	$$
	Indeed, given $g\in \ell_1(\mathcal{B})$, observe that
	$$
	g(f)=\sum_{x^*\in \mathcal{B}} g(x^*)f(x^*) \qquad  (f\in \ell_\infty(\mathcal{B}))
	$$
	and
	$$
	[G^*(g)](x)=g(Gx)=\sum_{x^*\in \mathcal{B}} g(x^*)(Gx)(x^*)=\sum_{x^*\in \mathcal{B}} g(x^*)x^*(x) \qquad (x\in X),
	$$
	so $G^*(g)=\sum_{x^*\in \mathcal{B}} g(x^*)x^*$. Now, fixed $x_0^*\in \mathcal{B}$, define $g_0\in \ell_\infty (\mathcal{B})$ by
	$$
	g_0(x^*)=\left\{
	\begin{array}{lr}
		1 \quad & \textnormal{ if } x^*=x_0^*\\
		0 \quad & \textnormal{ if } x^*\neq x_0^*
	\end{array}\right.
	$$
	which clearly satisfies $G^*(g_0)=x_0^*$. Therefore, by the arbitrariness of $x_0^*\in \mathcal{B}$, we get $G^*(B_{\ell_{\infty}(\mathcal{B})^*})\supset \mathcal{B}$, so $G^*(B_{\ell_{\infty}(\mathcal{B})^*})$ is a spear set and $G$ is generating by Corollary~\ref{corollary:charact-generating-dual-spear}.
\end{proof}

The above proof, when read pointwise, allows to give a characterization of those points at which every generating operator attains its norm.

\begin{proposition}\label{prop:charact-point-of-norm-attainment-Gen-op}
Let $X$ be a Banach space and $x_0 \in S_X$. Then, the following are equivalent:
\begin{enumerate}

\item[$(i)$] For every Banach space $Y$ and for every generating operator $G\in S_{\mathcal{L}(X,Y)}$ one has $\|Gx_0\|=1$.

\item[$(ii)$] The equality $\sup\limits_{x^*\in \mathcal{B}} |x^*(x_0)|=1$ holds for every spear set $\mathcal{B}\subseteq B_{X^*}$.
\end{enumerate}
\end{proposition}

\begin{proof}
$(ii)\Rightarrow(i)$ Given a generating operator $G$, $\mathcal{B}=G^*(B_{Y^*})$ is a spear set by Corollary~\ref{corollary:charact-generating-dual-spear} so
$$
\|Gx_0\|=\sup\limits_{y^*\in B_{Y^*}} |y^*(Gx_0)|=\sup\limits_{x^*\in \mathcal{B}} |x^*(x_0)|=1.
$$

$(i)\Rightarrow(ii)$ Suppose that $(ii)$ does not hold. Then, there is a spear set $\mathcal{B}\subseteq B_{X^*}$ such that $\sup\limits_{x^*\in \mathcal{B}} |x^*(x_0)|<1$. Now, the operator $G\colon X\longrightarrow\ell_\infty(\mathcal{B})$ defined by
	$$
	(Gx)(x^*)=x^*(x) \qquad (x^*\in X^*,\, x\in X)
	$$
	is generating (as shown in the proof of Theorem~\ref{theorem:char-when-X-domain-generating-NONA})
	and satisfies
	$$
	\|G(x_0)\|=\sup_{x^*\in \mathcal{B}} |(Gx_0)(x^*)|=\sup_{x^*\in \mathcal{B}} |x^*(x_0)|<1.
	$$
	Therefore, $(i)$ does not hold.
\end{proof}

\section{The set of all generating operators}\label{section:setGen(X,Y)}
Our aim here is to study the set $\Gen(X,Y)$ of all generating operators between the Banach spaces $X$ and $Y$. Recall, on the one hand, that $\Id_X\in \Gen(X,X)$ for every Banach space $X$, so $\Gen(X,X)\neq \emptyset$ for every Banach space $X$. On the other hand, recall that Corollary~\ref{cor:charact-generatin-rank1} shows that $\Gen(X,\K)=\Spear(X^*)$, so $\Gen(X,\K)$ is empty for many Banach spaces $X$: those for which $\Spear(X^*)=\emptyset$ as uniformly smooth spaces, strictly convex spaces, or real smooth spaces with dimension at least two (see \cite[Proposition~2.11]{KMMP-Spear}). We will be interested in finding conditions to ensure that $\Gen(X,Y)$ is non-empty and, in those cases, to study how big the set $\Gen(X,Y)$ can be. We start with an easy observation on $\Gen(X,Y)$.

\begin{proposition}\label{prop:GenXY-norm-closed}
Let $X$, $Y$ be Banach spaces. Then, $\Gen(X,Y)$ is norm-closed.
\end{proposition}

\begin{proof}
Fixed $G_0\in \overline{\Gen(X,Y)}$ and $n\in \N$, there is $G_n\in\Gen(X,Y)$ such that $\|G_0-G_n\|<1/n$ and, therefore, $\|G_0^*-G_n^*\|<1/n$. Observe now that, for $x^*\in X^*$, we have
\begin{align*}
\max_{\theta\in\T}\sup_{y^*\in B_{Y^*}} \|G_0^*(y^*)+\theta x^*\|
	&\geq \max_{\theta\in\T}\sup_{y^*\in B_{Y^*}}\|G_n^*(y^*)+\theta x^*\| -\sup_{y^*\in B_{Y^*}}\|(G_0^*-G_n^*)(y^*)\|
	\\ &=\|G_n^*(B_{Y^*})+\T x^*\|-\|G_0^*-G_n^*\|> 1+\|x^*\|-1/n,
\end{align*}
where the last inequality holds by Corollary~\ref{corollary:charact-generating-dual-spear} since $G_n$ is generating. Now, it follows again from Corollary~\ref{corollary:charact-generating-dual-spear} that $G_0\in\nolinebreak\Gen(X,Y)$.
\end{proof}

Next, we study the problem of finding out whether $\Gen(X,Y)$ is empty or not for the Banach spaces $X$ and $Y$ from two points of view: fixing the space $Y$ and fixing the space $X$.

\subsection{\texorpdfstring{$\boldsymbol{\Gen(X,Y)}$}{Gen(X,Y)} when \texorpdfstring{$\boldsymbol{Y}$}{Y} is fixed}
We will show that for every Banach space $Y$ there is another Banach space $X$ such that $\Gen(X,Y)=\emptyset$.

\begin{proposition}\label{proposition:foreveryYthereisXgenempty}
For every Banach space $Y$ there is a Banach space $X$ such that $\Gen(X,Y)=\emptyset$.
\end{proposition}

We need the following obstructive result for the existence of generating operators that will serve to our purpose.

\begin{lemma}\label{lemma:Gen-and-Frechet-diff}
Let $X$, $Y$ be Banach spaces and let $G\in \Gen(X,Y)$. If the norm of $X^*$ is Fr\'{e}chet differentiable at $x_0^*\in S_{X^*}$ and $x_0^*$ is strongly exposed, then $x_0^*\in \overline{G^*(B_{Y^*})}$.
\end{lemma}

\begin{proof}
Suppose that $x_0^*\notin \overline{G^*(B_{Y^*})}$ and let $\alpha=\dist(x_0^*,\overline{G^*(B_{Y^*})})>0$. Since $x_0^*$ is strongly exposed, there are $x\in S_{X}$ and $\delta>0$ satisfying
$ \re x_0^*(x)=1$ and $\diam(\Slice(B_{X^*},x,\delta))<\alpha$.
Therefore, we get $\re x^*(x)\leq 1-\delta$ for every $x^*\in \overline{G^*(B_{Y^*})}$ and, as $\overline{G^*(B_{Y^*})}$ is a balanced set, we get in fact that
\begin{equation}\label{eq:x-separa}
|x^*(x)|\leq 1-\delta \qquad \forall\ x^*\in \overline{G^*(B_{Y^*})}.
\end{equation}
By Corollary~\ref{corollary:charact-generating-dual-spear}, $\overline{G^*(B_{Y^*})}$ is a spear set, so we can find a sequence $\{x_n^*\}$ in $\overline{G^*(B_{Y^*})}$ and a sequence $\{\theta_n\}$ in $\T$  such that $\|\theta_nx_n^*+ x_0^*\|\to 2$. Therefore, there is a sequence $\{x_n\}$ in $S_{X}$ satisfying
$$
\re x_0^*(x_n)\to 1 \qquad \text{and} \qquad |x_n^*(x_n)|\to 1.
$$
Since the norm of $X^*$ is Fr\'{e}chet differentiable at $x_0^*\in S_{X^*}$, by the $\check{S}$mulyian's test, we have that $\|x_n-x\|\to 0$. Thus, we get $|x_n^*(x)|\to 1$ which contradicts \eqref{eq:x-separa}.
\end{proof}

We are now able to provide the pending proof. For a Banach space $X$ let $\dens(X)$ denote its density character.

\begin{proof}[Proof of Proposition~\ref{proposition:foreveryYthereisXgenempty}]
Take a set $\Lambda$ with cardinality greater than $\dens(Y^*)$ and set $X=\ell_2(\Lambda)$.
If $G\in\Gen(X,Y)$, it follows from Lemma~\ref{lemma:Gen-and-Frechet-diff} that $\overline{G^*(Y^*)}=X^*=\ell_2(\Lambda)$ since every point in $S_{X^*}$ is Fr\'{e}chet differentiable and strongly exposed. Then, $\dens (X^*)=\dens(\overline{G^*(Y^*)})\leq \dens(Y^*)$, which is a contradiction.
\end{proof}

The above argument is based on the possibility of considering Banach spaces in the domain with a very big density character. It is then natural to raise the following question.

\begin{question}
Does there exist a Banach space $Y$ with $\dens(Y)=\Gamma$ such that $\Gen(X,Y)\neq \emptyset$ for every Banach space $X$ satisfying $\dens(X)\leq \Gamma$?
\end{question}

This question is easily solvable for separable spaces. Indeed, the space $Y=C[0,1]$ contains isometrically every separable Banach space. Since isometric embeddings are generating, we get the following example.

\begin{example}\label{example:Y=C[01]everyXseparableGennotempty}
{\slshape The separable Banach space $Y=C[0,1]$ satisfies $\Gen(X,Y)\neq \emptyset$ for every separable Banach space $X$.}
\end{example}

The question of whether the same trick works for all density characters is involved and depends on the Axiomatic Set Theory. On the one hand, assuming CH, $\ell_\infty/c_0$ is isometrically universal for all Banach spaces of density character the continuum \cite{Parovicenko} but, on the other hand, it is consistent that no such a universal space exists \cite{ShelahUsvyatsov}, even a isomorphically universal space, see \cite{Brech-Koszmider-2012}.

\subsection{\texorpdfstring{$\boldsymbol{\Gen(X,Y)}$}{Gen(X,Y)} when \texorpdfstring{$\boldsymbol{X}$}{X} is fixed}

We start our discussion recalling that, by Corollary~\ref{cor:charact-generatin-rank1}, a rank-one operator $x^*\otimes y\in G(X,Y)$ is generating if and only if $x^*\in \Spear(X^*)$. This, together with the fact that $\Gen(X,\K)=\Spear(X^*)$, gives the following result.

\begin{corollary}\label{corollary:GenXYnotemptyforallYiffspearXstarnotempty}
Let $X$ be a Banach space. Then,
$$
\Gen(X,Y)\neq \emptyset \  \textnormal{ for every Banach space } Y \quad \Longleftrightarrow \quad \Spear(X^*)\neq \emptyset.
$$
\end{corollary}

For instance, if $X$ has the alternative Daugavet property and $B_{X^*}$ has $w^*$-denting points, then $\Spear(X^*)\neq\emptyset$ by \cite[Proposition 5.1]{KMMP-Spear}.

Once we know about the existence of Banach spaces for which $\Gen(X,Y)\neq \emptyset$ for every Banach space $Y$, it is natural to ask about the possible size of the set $\Gen(X,Y)$. The maximal possibility is $\Gen(X,Y)=S_{\mathcal{L}(X,Y)}$, but this forces $X=\K$.

\begin{corollary}\label{corollary:XonedimensionaliffGenXYequalall}
Let $X$ be a Banach space. Then, there exists a Banach space $Y$ such that $\Gen(X,Y)=S_{\mathcal{L}(X,Y)}$ if and only if $X=\K$. In this case, $\Gen(X,Z)=S_{\mathcal{L}(X,Z)}$ for all Banach spaces $Z$.
\end{corollary}

\begin{proof}
If $X=\K$ then $\Gen(X,Y)=S_{\mathcal{L}(X,Y)}$ obviously holds for every Banach space $Y$. Conversely, suppose that there is a Banach space $Y$ such that $\Gen(X,Y)=S_{\mathcal{L}(X,Y)}$. So, in particular, every rank-one operator in $S_{\mathcal{L}(X,Y)}$ is generating but this means that $\Spear(X^*)=S_{X^*}$ by Corollary~\ref{cor:charact-generatin-rank1}. Therefore, $X=\K$ by \cite[Proposition~2.11.(e)]{KMMP-Spear}
\end{proof}

It is now natural to wonder if there can be enough generating operators to recover the unit ball of $\mathcal{L}(X,Y)$ by convex (or closed convex) hull. That is, we are looking for Banach spaces $X$ such that $B_{\mathcal{L}(X,Y)}=\conv(\Gen(X,Y))$ or $B_{\mathcal{L}(X,Y)}=\overline{\conv}(\Gen(X,Y))$ for every Banach space $Y$.

We start our discussion with an observation on lush spaces. Recall that a Banach spaces $X$ is \emph{lush} \cite{BKMW-lush} if for every $x,y\in S_X$ and every $\eps>0$, there exists $y^*\in S_{Y^*}$ such that $y\in \Slice(B_X,y^*,\eps)$ and $\dist(x,\aconv(\Slice(B_X,y^*,\eps)))<\eps$. Observe that $B_{X^*}=\overline{\conv}^{w^*}(\Gen(X,\K))=\overline{\conv}^{w^*}(\Spear(X^*))$ implies that $X$ is lush by \cite[Proposition 3.32]{KMMP-Spear}. Conversely, if $X$ is lush and separable, then $B_{X^*}=\overline{\conv}^{w^*}(\Gen(X,\K))$ by \cite[Theorem~3.33]{KMMP-Spear}. If one replaces the weak-star closed convex hull by the norm closed convex hull, one gets some interesting results on almost CL-spaces. A Banach space $X$ is said to be an \emph{almost CL-space} \cite{Lima} if $B_X$ is the absolutely closed convex hull of every maximal convex subset of $S_X$. By Hahn-Banach and Krein-Milman theorems, every maximal convex subset of $S_X$ has the form $\Face(B_X,x^*):=\{x\in S_X\colon x^*(x)=1\}$ for suitable $x^*\in\ext(B_{X^*})$. In this case, we say that $x^*$ is a \emph{maximal extreme point}, and write $x^*\in \extm(B_{X^*})$.

\begin{proposition}\label{prop:Bola=conv-gen-X-K-implies-X*-almostCL}
Let $X$ be a Banach space satisfying that $B_{X^*}=\overline{\conv}(\Gen(X,\K))$. Then, $X^*$ is an almost CL-space.
\end{proposition}

\begin{proof}
Indeed, let $F=\Face(S_{X^*},x^{**})$ for some $x^{**}\in \extm(B_{X^{**}})$ be a maximal convex subset of $S_{X^*}$. Then, $B_{X^*}=\overline{\conv}(\T F)$ since
$\Spear(X^*)\equiv \Gen(X,\K) \subseteq \T \Face(S_{X^*},x^{**})$ for all $x^{**}\in \ext(B_{X^{**}})$ by \cite[Corollary 2.8.iv]{KMMP-Spear}.
\end{proof}

A partial converse of the above result is also true:

\begin{proposition}
Let $X$ be an almost CL-space. Then, $B_{X^*}=\overline{\conv}^{w^*}(\Gen(X,\K))$. If, moreover, $X$ does not contain $\ell_1$, then
$B_{X^*}=\overline{\conv}(\Gen(X,\K))$.
\end{proposition}

\begin{proof}
Being $\extm(B_{X^*})$ norming for $X$, we always have that
$$B_{X^*}=\overline{\conv}^{w^*}(\extm(B_{X^*})).$$ But when $X$ is an almost CL-space,
we have that $|x^{**}(x^*)|=1$ for every $x^{**}\in\ext(B_{X^{**}})$ and every $x^*\in\extm(B_{X^*})$ by using \cite[Lemma 3]{MP-CL-spaces}. Then, $\extm(B_{X^*})\subseteq \Spear(X^*)\equiv \Gen(X,\K)$ by \cite[Corollary 2.8.iv]{KMMP-Spear}, and we are done.

For the moreover part, it is enough to see that $\extm(B_{X^*})$ is actually a James boundary for $X$ and so $B_{X^*}=\overline{\conv}(\extm(B_{X^*}))$ by \cite[Theorem~III.1]{Godefroy-boundary}.
\end{proof}

Our next aim is to show that the set $\Gen(L_1(\mu), Y)$ is quite big for every finite measure $\mu$ and many Banach spaces $Y$, and that in some cases it allows to recover the unit ball of $\mathcal{L}(L_1(\mu),Y)$ by taking closed convex hull. Given a finite measure space $(\Omega,\Sigma,\mu)$ and a Banach space $Y$ we write
$$
\mathcal{R}(L_1(\mu),Y)=\{T\in \mathcal{L}(L_1(\mu), Y) \colon \|T\|\leq 1,\ T \textnormal{ is representable}\}.
$$

\begin{theorem}\label{theorem:representable-in-Gen}
Let $(\Omega,\Sigma,\mu)$ be a finite measure space and let $Y$ be a Banach space. Then,
$$
\mathcal{R}(L_1(\mu),Y)\subseteq \overline{\conv}\big(\Gen(L_1(\mu),Y)\big).
$$
As a consequence, if $Y$ has the RNP, then
$$
B_{\mathcal{L}(L_1(\mu), Y)}=\overline{\conv}\big(\Gen(L_1(\mu),Y)\big).
$$
\end{theorem}

Observe that the restriction on the measure $\mu$ to be finite can be relaxed to be $\sigma$-finite as in Remark~\ref{remark:finite-sigma-finite}.

The proof of the theorem follows immediately using Corollary~\ref{cor:char-generating-from-L1} and the next lemma, which we do not know whether it is already known.

\begin{lemma}\label{lemma:L_infty(mu,Y)}
Let $(\Omega,\Sigma,\mu)$ be a positive measure space and let $Y$ be a Banach space. Then,
$$
B_{L_\infty(\mu,Y)}=\overline{\conv}\big(\{g\in L_\infty(\mu,Y) \colon \|g(t)\|=1\ \mu\textnormal{-almost everywhere}\}\big).
$$
\end{lemma}

\begin{proof}\parskip=0ex
Calling $\mathcal{B}=\{g\in L_\infty(\mu,Y) \colon \|g(t)\|=1\ \mu\textnormal{-almost everywhere}\}$,  it obviously suffices to show that
$S_{L_\infty(\mu,Y)}\subset\overline{\conv}(\mathcal{B})$. We divide the proof into two steps.

\emph{Step one.}\ {\slshape Let $f\in S_{L_\infty(\mu,Y)}$ and suppose that there are $N\in \N$, numbers $\alpha_1<\dots<\alpha_N \in [0,1]$, and pairwise disjoint subsets $B_k\subset \Omega$ with $\mu(B_k)\neq 0$ for $k=1,\ldots,N$ such that $\bigcup_{k=1}^NB_k=\Omega$ and $\|f(t)\|=\alpha_k$ for every $t\in B_k$ and every $k=1,\ldots,N$ (observe that $\alpha_N=1$ as $\|f\|=1$). Then, $f$ can be written as a convex combination of $2^{N-1}$ functions in $\mathcal{B}$.}\newline
\indent Indeed, we proceed by induction on $N$: for $N=1$, the function $f$ belongs to $\mathcal{B}$. The case $N=2$ gives the flavour of the proof. In this case we have that $\|f(t)\|=\alpha_1$ for every $t\in B_1$ and $\|f(t)\|=1$ for every $t\in B_2$. So, call $\lambda_1=\frac{1+\alpha_1}{2}$, $\lambda_2=\frac{1-\alpha_1}{2}\in[0,1]$ and define $g_1, g_2 \in L_\infty(\mu,Y)$ by $g_1(t)=g_2(t)=f(t)$ for every $t\in B_2$. Besides, if $\alpha_1\neq0$, define
$$
g_1(t)=\frac{f(t)}{\|f(t)\|}, \quad \textnormal{and} \quad g_2(t)=-\frac{f(t)}{\|f(t)\|} \quad \forall t\in B_1.
$$
If otherwise $\alpha_1=0$, fix $y_0\in S_Y$, and define $g_1(t)=y_0$ and $g_2(t)=-y_0$ for every $t\in B_1$. It is clear that in any case we have $f=\lambda_1g_1+\lambda_2g_2$ and that $g_1,g_2\in \mathcal{B}$.

Suppose now that the result is true for $N\geq 2$ and let us prove it for $N+1$. So, let $f\in S_{L_\infty(\mu,Y)}$ and suppose that there are numbers $\alpha_1<\dots<\alpha_{N+1} \in [0,1]$ with $\alpha_{N+1}=1$, and pairwise disjoint subsets $B_k\subset \Omega$ with $\mu(B_k)\neq 0$ for $k=1,\ldots,N+1$ such that $\bigcup_{k=1}^{N+1}B_k=\Omega$ and $\|f(t)\|=\alpha_k$ for every $t\in B_k$ and every $k=1,\ldots,N+1$. Observe that, as $N\geq 2$, we have that $\alpha_N> 0$. Then, we call $\lambda_1=\frac{1+\alpha_N}{2}$, $\lambda_2=\frac{1-\alpha_N}{2}\in[0,1]$ and we define $f_1,f_2\in L_\infty(\mu,Y)$ by
\begin{align*}
f_1(t)&=\frac{f(t)}{\|f(t)\|} \quad \textnormal{if } t\in B_{N} \quad \textnormal{and} \quad f_1(t)=f(t) \quad \textnormal{if } t\in \Omega\setminus B_{N},\\
f_2(t)&=-\frac{f(t)}{\|f(t)\|} \quad \textnormal{if } t\in B_{N} \quad \textnormal{and} \quad f_2(t)=f(t) \quad \textnormal{if } t\in \Omega\setminus B_{N}
\end{align*}
which clearly satisfy $f=\lambda_1 f_1+\lambda_2f_2$. Besides, it is also clear that  $\|f_1(t)\|=\|f_2(t)\|=1$ for every $t\in B_N\cup B_{N+1}$. So, we can apply the induction step for $f_1$ and $f_2$ to write
$$
f_1=\sum_{k=1}^{2^{N-1}} \mu_k g_k \qquad \textnormal{and} \qquad f_2=\sum_{k=1}^{2^{N-1}} \beta_k h_k
$$
where $g_k,h_k\in \mathcal{B}$,  $\mu_k,\beta_k\in [0,1]$ for $k=1,\ldots, 2^{N-1}$,  $\sum_{k=1}^{2^{N-1}} \mu_k=1$, and $\sum_{k=1}^{2^{N-1}} \beta_k=1$. Therefore, the convex combination we are looking for is
$$
f=\lambda_1\sum_{k=1}^{2^{N-1}} \mu_k g_k+ \lambda_2\sum_{k=1}^{2^{N-1}} \beta_k h_k
$$
which finishes the induction process.

\emph{Step two.} {\slshape Every function $f\in S_{L_\infty(\mu,Y)}$ can be approximated by functions of the class described in the first step.}\newline \indent
Indeed, fixed $\eps>0$, we may find a partition of $[0,1]=\bigcup_{k=1}^{N} A_k$ such that $0<\diam(A_k)<\eps$ for every $k=1,\ldots,N$, $0\in A_1$, and $1\in A_N$. Next, fix $\alpha_k\in A_k$ for each $k=1,\ldots,N$ with $\alpha_1=0$ and $\alpha_N=1$, and define $B_k=\{t\in\Omega\colon \|f(t)\|\in A_k\}$ for every $k=1,\ldots,N$. We assume without loss of generality that $B_1,\ldots,B_N$ are non-empty. Now, consider the function $h\in L_\infty(\mu,Y)$ given by
$$
h(t)=
\begin{cases}
	0  &\textnormal{if } t\in B_1
	\\
	\alpha_k \dfrac{f(t)}{\|f(t)\|} &\textnormal{if } t\in B_k \textnormal{ with } k\geq 2.
\end{cases}
$$
For $t\in B_1$, we have
$$
\|f(t)-h(t)\|=\|f(t)\|\leq\diam(A_1)<\eps.
$$
Besides, for $t\in B_k$ with $k\geq 2$, we have
$$
\|f(t)-h(t)\|=\left\|f(t)-\alpha_k \dfrac{f(t)}{\|f(t)\|}\right\|=\big|\|f(t)\|-\alpha_k\big|\leq \diam(A_k)<\eps.
$$
Therefore, $\|f-h\|\leq\eps$ and the proof is finished.
\end{proof}

Let us now discuss the case of purely atomic measures. When $\mu$ is purely atomic and $\sigma$-finite (so $L_1(\mu)$ can be easily viewed as $L_1(\nu)$ for a suitable purely atomic and finite measure $\nu$, see \cite[Proposition~1.6.1]{CembranosMendoza} for instance), every operator in $\mathcal{L}(L_1(\mu), Y)$ is representable for every Banach space $Y$ (see \cite[p.~62]{DiestelUhl}, for instance). So, Theorem~\ref{theorem:representable-in-Gen} gives that $B_{\mathcal{L}(\ell_1(\Gamma),Y)}=\overline{\conv}(\Gen(\ell_1(\Gamma),Y))$ for every Banach space $Y$ and every countable set $\Gamma$. Actually, the restriction of countability for the set $\Gamma$ can be remove and the proof in this case is much more direct.

\begin{proposition}\label{cor:conv-Gen-l1}
$B_{\mathcal{L}(\ell_1(\Gamma),Y)}=\overline{\conv}(\Gen(\ell_1(\Gamma),Y))$ for every Banach space $Y$ and every set $\Gamma$.
\end{proposition}

\begin{proof}
The space $\mathcal{L}(\ell_1(\Gamma), Y)$ can be easily identified with $\left[\bigoplus_{\gamma\in \Gamma}Y\right]_{\ell_\infty}$ using the isometric isomorphism  $\Phi\colon\mathcal{L}(\ell_1(\Gamma), Y)\longrightarrow\left[\bigoplus_{\gamma\in \Gamma}Y\right]_{\ell_\infty}$ given by  $\Phi(T)= (Te_\gamma)_{\gamma\in\Gamma}$ (see the proof of \cite[Lemma~2]{PayaSaleh}, for instance). With this identification and Example~\ref{exa:ell_1}, generating operators in $\mathcal{L}(\ell_1(\Gamma), Y)$ are exactly elements in $\left[\bigoplus_{\gamma\in \Gamma}Y\right]_{\ell_\infty}$ with every coordinate having norm one. Therefore, Lemma~\ref{lemma:L_infty(mu,Y)} gives the result.
\end{proof}

For finite-dimensional $\ell_1$-spaces, we get a better result.

\begin{corollary}
$B_{\mathcal{L}(\ell_1^n,Y)}=\conv(\Gen(\ell_1^n,Y))$ for every Banach space $Y$ and every $n\in \N$.
\end{corollary}

\begin{proof}
For $T\in B_{\mathcal{L}(\ell_1^n,Y)}$ consider the finite-dimensional subspace of $Y$ given by $Y_1=T(\ell_1^n)$ and observe that $\overline{\conv}(\Gen(\ell_1^n,Y_1))=\conv(\Gen(\ell_1^n,Y_1))$ as $\Gen(\ell_1^n,Y_1)$ is compact. So, Proposition~\ref{cor:conv-Gen-l1} tells us that
$$
T\in \conv(\Gen(\ell_1^n,Y_1)).
$$
Finally, denoting $G_1$ the inclusion of $Y_1$ in $Y$, it is obvious that $G_1\circ G\in \Gen(\ell_1^n,Y)$ for every $G\in \Gen(\ell_1^n,Y_1)$. So $T\in \conv(\Gen(\ell_1^n,Y))$.
\end{proof}

The next result shows that the only finite-dimensional real spaces with this property are $\ell_1^n$ for $n\in \N$.

\begin{proposition}\label{thm:conv-Gen-X-finite-ext}
Let $X$ be a \emph{real} Banach space with $\dim(X)=n$ and such that $B_{\mathcal{L}(X,Y)}=\overline{\conv}(\Gen(X,Y))$ for every Banach space $Y$. Then, $X=\ell_1^n$.
\end{proposition}

\begin{proof}\parskip=0ex
Proposition~\ref{prop:Bola=conv-gen-X-K-implies-X*-almostCL} tells us that $X^*$ is an almost CL-space so $n(X^*)=n(X)=1$. Therefore, as $X$ is real, the set $\ext(B_X)$ is finite by \cite[Theorem~3.2]{McGregor}. Our goal is to show that $\ext(B_X)$ contains exactly $2n$ elements as this clearly implies that $X$ is isometrically isomorphic to the real space $\ell_1^n$.

We suppose that $\ext(B_X)$ has more than $2n$ elements and we show that, in such a case, there is a Banach space $Y$ ($=X$ with a new norm) such that $B_{\mathcal{L}(X,Y)}\neq \conv(\Gen(X,Y))$. Since $\dim(X)=n$ and $\ext(B_X)$ has more than $2n$ elements, we may find $\{e_1,\ldots,e_n\} \subset\ext(B_X)$ linearly independent and $e_{n+1} \in\ext(B_X)$ satisfying
$$
e_{n+1} \notin \{\pm e_j \colon j=1,\ldots,n\}.
$$
For each $j=1,\ldots,n$, as $\ext(B_X)$ is finite, we can pick $f_j\in X^*$ such that
$$
1=f_j(e_j)>c_j=\max\left\{f_j(x)\colon x\in\ext(B_X)\setminus\{e_j\} \right\}.
$$
Besides, define $c=\max\left\{c_j\colon j=1,\ldots,n\right\}<1$, take $\eps>0$ satisfying $(1+\eps)c<1$, and consider the Banach space $Y$ whose unit ball is
$$
B_Y=\conv \Big(\ext(B_X)\cup \{\pm(1+\eps)e_{n+1}\}\Big).
$$
Now, observe that $e_1,\ldots,e_n$ are also extreme points of $B_Y$. Indeed, fixed $j\in \{1,\ldots,n\}$, our choice of $c$ gives
$$
f_j(x)\leq (1+\eps)c_j<1=f_j(e_j)
$$
for every $x\in \ext(B_X)\cup \{\pm(1+\eps)e_{n+1}\}$ with $x\neq e_j$. So $e_j$ cannot lie in a proper segment of $B_Y$.

Observe that $\overline{\conv}(\Gen(X,Y))=\conv(\Gen(X,Y))$, as $\mathcal{L}(X,Y)$ is finite-dimensional and $\Gen(X,Y)$ is norm-closed by Proposition~\ref{prop:GenXY-norm-closed}.

Finally, consider the operator $\Id\in \mathcal{L}(X,Y)$ which is not generating by Corollary~\ref{cor:X-finite-dimension} because $e_{n+1}\in \ext(B_X)$ and $\|\Id(e_{n+1})\|_Y=\|e_{n+1}\|_Y<1$. If $\Id \in\conv(\Gen(X,Y))$, we may find $M\in \N$, $\lambda_1,\ldots,\lambda_M\geq 0$ with $\sum_{i=1}^M \lambda_i=1$ and $G_1,\ldots,G_M\in \Gen(X,Y)$ such that
$$
\Id=\sum_{i=1}^M\lambda_i G_i.
$$
Then, for each $j=1,\ldots,n$, we have that
$$
e_j=\Id(e_j)=\sum_{i=1}^M \lambda_i G_i(e_j)\quad \Longrightarrow \quad G_i(e_j)=e_j \quad \forall \, i\in\{1,\ldots,M\}
$$
as $e_j\in \ext(B_Y)$. Since $\{e_1,\ldots,e_n\}$ is linearly independent and $\dim(X)=n$, it follows that $G_i=\Id$ for all $i=1,\ldots,M$. Therefore, we have that $\Id \notin\conv(\Gen(X,Y))=\overline{\conv}(\Gen(X,Y))$ which finishes the proof.
\end{proof}

\section*{Acknowledgments}
The authors would like to thank Antonio Avil\'{e}s and Rafael Pay\'{a} for kindly answering several questions regarding the topics of this manuscript.

Part of the research of this manuscript was done while Alicia Quero was visiting V.~N. Kazarin National University in Kharkiv, Ukraine, from March to May 2021 with the support of the Spanish Ministerio de Universidades, grants FPU18/03057 and EST19/00601. The authors have been supported by PID2021-122126NB-C31 funded by MCIN/AEI/ 10.13039/501100011033 and ``ERDF A way of making Europe'', by Junta de Andaluc\'ia I+D+i grants P20\_00255 and FQM-185, and by ``Maria de Maeztu'' Excellence Unit IMAG, reference CEX2020-001105-M funded by MCIN/AEI/10.13039/501100011033.

\end{document}